\documentclass[11pt, a4paper, onesided]{preprint}
\pdfoutput=1 %

\usepackage{tikz}
\usetikzlibrary{matrix, positioning}

\usepackage{nicematrix}
\usepackage{multirow}
\usepackage{multibib}
\newcites{soft}{Mathematical software}

\newcommand{\del}{\mathbin\backslash}
\newcommand{\con}{\mathbin/}

\newcommand{\Sym}{\RM{Sym}}
\newcommand{\PR}{\RM{PR}}
\newcommand{\PD}{\RM{PD}}
\newcommand{\pmm}{\RM{pr}}
\newcommand{\LPR}{\RM{LPR}}

\newcommand{\spos}{\TT+}
\newcommand{\sneg}{\TT-}
\newcommand{\SL}{\RM{SL}}
\newcommand{\SymGrp}{\FR{S}}
\newcommand{\HypGrp}{\FR{B}}
\DeclareMathOperator{\sgn}{sgn}
\DeclareMathOperator{\signs}{signs}

\DeclareMathOperator{\diag}{diag}
\DeclareMathOperator{\adj}{adj}

\newcommand\pr[1]{
  \setsepchar{:}
  \ignoreemptyitems
  \readlist*\mylist{#1}%
  \ifthenelse{\mylistlen = 2}{\mylist[2]}{}%
  [\mylist[1]]%
}

\newcommand\apr[1]{
  \setsepchar{:}
  \ignoreemptyitems
  \readlist*\mylist{#1}%
  \ifthenelse{\mylistlen = 2}{\mylist[2]}{}%
  [\mylist[1]]%
}

\title[Sign patterns of principal minors of real symmetric matrices]{%
Sign patterns of principal minors\\%
of real symmetric matrices}

\author{Tobias Boege}
\address{\mbox{Tobias Boege, Department of Mathematics, KTH Royal Institute of Technology, Sweden}}
\email{post@taboege.de}

\author{Jesse Selover}
\address{\mbox{Jesse Selover, Department of Mathematics, University of Massachusetts Amherst, Amherst, USA}}
\email{jselover@umass.edu}

\author{Maksym Zubkov}
\address{\mbox{Maksym Zubkov, Department of Mathematics, University of California, Berkeley, USA}}
\email{mzubkov@berkeley.edu}

\date{\today}

\newcommand{\MathRepoURL}{\url{https://mathrepo.mis.mpg.de/SymmetricPrincipalMinorSigns/}}



\begin{document}

\begin{abstract}
We analyze a combinatorial rule satisfied by the signs of principal minors of
a real symmetric matrix. The sign patterns satisfying this rule are equivalent
to uniform oriented Lagrangian matroids. We first discuss their structure and
symmetries and then study their asymptotics, proving that almost all of them
are not representable by real symmetric matrices. We offer several conjectures
and experimental results concerning representable sign patterns and the topology
of their representation spaces.
\end{abstract}

\maketitle

\section{Introduction}
\label{sec:Introduction}

Let $N$ be a finite set of size $n$ and $\Sigma \in \Sym_N \defas \Sym_N(\BB R)$
a real symmetric matrix whose rows and columns are indexed by~$N$. For any subsets
$I,J \subseteq N$ denote by \(\Sigma_{I,J}\) the submatrix whose rows are indexed
by \(I\) and columns by \(J\). A submatrix \(\Sigma_K \defas \Sigma_{K,K}\) is
called a \emph{principal submatrix}, and we write $\pr{K:\Sigma} \defas \det \Sigma_K$
for the \(K\)-th \emph{principal minor} of \(\Sigma\).

In this article we are concerned with the space $\PR_N \defas \PR_N(\BB R)$
of real symmetric matrices all of whose principal minors are non-zero.
Such a matrix is called \emph{principally regular}. The space $\PR_N$
appears in a variety of settings. We delve into these connections more in
the \nameref{sec:Background} section below. $\PR_N$ has a natural
decomposition whose pieces correspond to assignments of \emph{signs}
$\TT+$ or $\TT-$ to each of the principal minors.
Formally, a \emph{sign pattern} is a function $s: 2^N \to \Set{\TT+, \TT-}$.
If~there exists $\Sigma \in \PR_N$ such that $s(K) = \sgn \pr{K:\Sigma}$
for all $K \subseteq N$, then $s$ is called \emph{representable} and we
denote it as~$\signs_\Sigma$. By~convention $\pr{\emptyset:\Sigma} = 1$
and hence $\signs_\Sigma(\emptyset) = \TT+$. The set $\Set{\TT+, \TT-}$
is naturally a group under the usual rules for multiplication of signs.

To work with sign patterns as combinatorial objects, we adopt some notational
conveniences. Subsets of $N$ are denoted by uppercase letters $I, J, K, \dots$
and elements correspondingly in lowercase $i, j, k, \dots$. The union of
$I, J \subseteq N$ is abbreviated to $IJ \defas I \cup J$. An~element $i \in N$
is not distinguished from its singleton subset $\Set{i} \subseteq N$ and thus
an expression $iK$ is shorthand for $\Set{i} \cup K$. We use $I \oplus J$ for
the symmetric difference of $I$ and $J$ and $K^\co$ for the complement of $K$
in~$N$. The notation $2^N$ for the powerset of~$N$ already appeared above.

For any sign pattern $s: 2^N \to \Set{\TT+, \TT-}$, we set \(\PR_N(s)\)
to be the fiber of \(\signs\) over $s$. That is,
\[
  \PR_N(s) \defas \Set{\Sigma\in \PR_N : \signs_\Sigma = s}.
\]
Topologically, \(\PR_N\) decomposes into the subspaces \(\PR_N(s)\), where
$s$ ranges over all representable sign patterns; and these subspaces are
disconnected from each other inside~$\PR_N$. Our aim is to understand the
pieces $\PR_N(s)$ as a window into the topology of the whole space $\PR_N$.

Our point of departure is the following observation about the minors
of a symmetric matrix. It follows at once from the Desnanot--Jacobi
identity and the symmetry of the matrix.

\begin{lemma*}
For any $\Sigma \in \Sym_N$, $i \neq j$ and $K \subseteq N \setminus ij$,
the following polynomial identity holds:
\begin{align*}
  \tag{$\diamond'$} \label{eq:Mdiamond}
  (\det \Sigma_{iK,jK})^2 = \pr{iK:\Sigma} \cdot \pr{jK:\Sigma} - \pr{ijK:\Sigma} \cdot \pr{K:\Sigma}.
\end{align*}
Consequently $\big[\signs_\Sigma(iK) \neq \signs_\Sigma(jK)\big] \implies
\big[\signs_\Sigma(K) \neq \signs_\Sigma(ijK)\big]$.
\end{lemma*}

The equation \eqref{eq:Mdiamond} is called a ``square trinomial'' in~\cite{Geometry}.
It shows that the principal minor map $K \mapsto \pr{K:\Sigma}$ is locally
log-submodular for a symmetric matrix, a fact also known as the
\emph{Koteljanskii inequality}. %
Motivated by its implication for the signs, we make the following definition.

\begin{definition} \label{def:Admissible}
A sign pattern $s: 2^N \to \Set{\TT+,\TT-}$ is \emph{admissible} if
$s(\emptyset) = \spos$ and if the implication
\begin{align*}
  \tag{$\diamond$} \label{eq:diamond}
  \big[s(iK) \neq s(jK)\big] \implies \big[s(K) \neq s(ijK)\big]
\end{align*}
holds for every \(i \neq j\) and \(K \subseteq N \setminus ij\).
\end{definition}

There is one instance of \eqref{eq:diamond} for every disjoint pair of
subsets $ij$ and $K$ of $N$, where ${i \neq j}$. We~denote such a
pair by $(ij|K)$ and call it a \emph{diamond} because the corresponding
interval $[K, ijK]$ in the boolean lattice $2^N$ looks like a diamond.
Correspondingly, we call the implication~\eqref{eq:diamond} the \emph{diamond
axiom} for~$(ij|K)$. Diamonds are in bijection with the $2$-dimensional
faces of the $N$-dimensional hypercube where the four sets $K, iK, jK, ijK$
in $[K, ijK]$ correspond to the four vertices of the $2$-face, so there
are precisely $\binom{n}{2} 2^{n-2}$ diamond axioms over a ground set
of size~$n$.
This also explains the name ``square trinomial'' for~\eqref{eq:Mdiamond}.

\subsection*{Background}
\label{sec:Background}

The principal minors of symmetric matrices have naturally received a lot
of attention from various branches of mathematics. Basic combinatorial
properties have been derived in the form of constraints on the
\emph{characteristic sequences} in \cite{pr,epr,sepr}. These sequences
capture information about principal minors $\pr{K:\Sigma}$ grouped by
the size of~$K$ which is called the \emph{order} of the principal minor.
The~\TT{sepr}-characteristic sequence of \cite{sepr} is the closest of
these notions to our topic of principal minor sign patterns. It captures
for each $0 \le k \le n$ whether all, some or none of the order $k$
principal minors are positive, negative, zero or non-zero. Crucially,
the theory of characteristic sequences makes no assumptions about
principal regularity but contains it as a special case.

In algebraic statistics and probabilistic reasoning, principally regular
matrices are algebraic models of the covariance matrices of regular
multivariate Gaussian probability distributions. They allow an algebraic
investigation of abstract properties of Gaussian conditional independence
relations using the tools of computer algebra~\cite[Chapter~3]{TB}.
In fact, the conditional independence relation of a Gaussian is nothing
but the set of all diamonds $(ij|K)$ such that equality holds in the
Koteljanskii inequality $\pr{iK:\Sigma}\cdot \pr{jK:\Sigma} \ge \pr{ijK:\Sigma}\cdot \pr{K:\Sigma}$
(provided that~$\pr{K:\Sigma} \not= 0$).

Principally regular matrices make yet another appearance in the standard
symplectic vector space $V = \BB R^n \oplus (\BB R^n)^\dual$. Recall that
the Lagrangian Grassmannian $\RM{LGr}(V)$ parametrizes the Lagrangian
subspaces of~$V$ and that a Lagrangian subspace can be represented as
the row space of a real $n \times 2n$-matrix $M = (A \mid B)$ of full
row rank, where $A$ and $B$ are $n \times n$-matrices with $A B^\T$
symmetric.
By \cite[Sections~3.4 \& 4.1.3]{CoxeterMatroids}, the Lagrangian Grassmannian
has a stratification by representable Lagrangian matroids. We consider its
affine patch consisting of matrices of the form $(I_n \mid \Sigma)$, where
$I_n$ is the identity matrix. This patch is parametrized by $\Sigma \in \Sym_N$
and $\PR_N$~is the cell corresponding to the uniform Lagrangian matroid.
This cell is further stratified by \emph{orientations} which are in turn
indexed by our representable sign patterns.
Our investigation of admissible sign patterns as a combinatorial model
for representable sign patterns is therefore in line with matroid theory:
oriented matroids \cite{OrientedMatroids} are combinatorial models for
the signs attainable by maximal minors of a matrix; and oriented gaussoids
\cite{Geometry} model signs of almost-principal minors of a positive
definite matrix. In fact, our admissible sign patterns are equivalent
to orientations of the uniform Lagrangian matroid; cf.~\cite[Axiom~4]{B.B.G+2000}.

In its role as a moduli space of Lagrangian subspaces with certain
orientation features, the topological structure of $\PR_N(s)$ becomes
interesting. Akin to the famous Ringel isotopy problem for oriented matroids
\cite[Section~8.6]{OrientedMatroids}, one may ask if every
$\Sigma \in \PR_N(s)$ can be continuously deformed into any other
$\Sigma' \in \PR_N(s)$ without leaving~$\PR_N(s)$. It turns out that this
is, in general, impossible because $\PR_N(s)$ is disconnected.
In light of our \Cref{conj:Contractible} that each connected component
is contractible, we focus on their number of connected components~$\dim H^0(\PR_N(s))$.

Our construction of the sign pattern $\signs_\Sigma$ factors through the
\emph{principal minor map}
\[
  \pmm_N: \Sym_N \to \mathbb{R}^{2^N}
\]
which sends a real symmetric matrix $\Sigma$ to the collection $(p_K =
\pr{K:\Sigma} : K \subseteq N)$ of all of its principal minors. This is a
polynomial map and hence, by the Tarski--Seidenberg theorem, its image is
a semialgebraic set.
This set is well-studied. We mention the thesis of Oeding \cite{Oeding} as
a starting point. More recently, Ahmadieh and Vinzant \cite{A.V2022} gave
a uniform algebraic characterization of its image over all unique factorization
domains. Over the real numbers, an equivalent description was already known
in the information theory literature due to Hassibi and Shadbakht \cite{S.H2011}
following work of Holtz and Sturmfels \cite{HoltzSturmfels}.

These descriptions of the image of $\pmm_N$, however, involve a large number
of complicated equations among the principal minors and only a few very
elementary inequalities $p_{ij} \le p_i\, p_j$ for all $i \not= j$.
Although there are many interesting inequalities on the image of $\pmm_N$
with applications to optimization, information theory and combinatorics
(see e.g.~\cite{Hyperbolic,hall2008bounded,S.H2011,InfoDet}), they are difficult
to derive from this description.
One source of motivation to study representability of sign patterns is
as combinatorial shadows of such inequalities. Namely, every sign pattern
$s$ corresponds to an open orthant in the space $\BB R^{2^N}$. If $s$
is non-representable, then the image of $\pmm_N$ does not intersect its
orthant. But more is true: from the non-representability it follows that
there exist $K_1, \dots, K_t, L \subseteq N$ such that every $\Sigma \in
\PR_N$ satisfies:
\[
  \bigwedge_{i=1}^t \Big[\sgn \pr{K_i:\Sigma} = s(K_i)\Big]
  \implies \Big[\sgn \pr{L:\Sigma} = -s(L)\Big].
\]
By the Positivstellensatz \cite[Proposition~4.4.1]{RealAlgebra}, the fact
that $\sgn \pr{L:\Sigma}$ can only be $-s(L)$ under the above assumptions
on $\sgn\pr{K_i:\Sigma}$ has an algebraic certificate: it is provable from
a polynomial identity among principal minors of symmetric matrices with
sums of squares coefficients --- just like~\eqref{eq:Mdiamond}.
Under the sign assumptions on the $\pr{K_i:\Sigma}$, such an identity turns
into a polynomial inequality which semialgebraically separates $\pmm_N(\PR_N)$
from the orthant associated to~$s$ (and possibly other orthants as well).
Knowing which $s$ are non-representable is a first step towards finding
these inequalities.
\Cref{conj:minimal-forbidden-minor} presents a candidate for the smallest
admissible sign pattern which is not representable and a proof of
non-representability might translate into one of the inequalities we seek.

\subsection*{Results and outline}
In \Cref{sec:Constructions} we introduce various combinatorial operations on
sign patterns which are inspired by concepts from matroid theory. The assumption
of principal regularity furnishes a particularly rich combinatorial structure.
We show that admissible sign patterns are structurally similar to representable~ones,
in that both classes satisfy natural closure properties with respect to these
operations. Moreover, these combinatorial constructions induce continuous
maps which give insight into the topology of $\PR_N(s)$ (\Cref{thm:ActionsTop}).
\Cref{sec:PR3} uses these combinatorial gadgets to give a full account of the
representable sign patterns in $\PR_3$ and their fibers under the $\signs$ map.

It is natural to ask whether admissibility is not only necessary for but even
equivalent to representability. In \Cref{sec:Counting} we derive the asymptotic
number of admissible sign patterns by identifying them with certain 3-colorings
of the hypercube. It follows that asymptotically almost all admissible sign
patterns are non-representable (\Cref{thm:asymp-zero-prob}).
Finally, in \Cref{sec:Representability} we study representability of sign patterns
for small ground sets and collect observations about the topology of~$\PR_N(s)$.
We~compute the numbers of admissible and representable sign patterns for $n \le 5$,
giving new sequences of combinatorial interest. Generalizations and further
topics, including the computational complexity of the representability problem,
are discussed briefly in \Cref{sec:Remarks}.
All~code and data referenced throughout the paper is available on our
supplementary repository
\begin{center}
  \MathRepoURL.
\end{center}

\section{Minors, duality and symmetry}
\label{sec:Constructions}

In this section we introduce various combinatorial constructions on sign
patterns which are inspired by matroid theory and come from natural
operations on $\PR_N$. The fact that admissible sign patterns also have
these closure properties motivates their use as a combinatorial model.

\subsection{Minors and duality}

In this section we derive combinatorial analogues to the three operations
of taking a principal submatrix, taking Schur complements, and matrix
inversion on the level of sign patterns. These three operations are
intimately connected by standard results in matrix analysis. The development
of the theory in this section follows the templates of matroid and gaussoid
theory.

Let $\Sigma \in \PR_N$. Then, by definition, every principal submatrix
$\Sigma_K$ is non-singular. Consider $N$ partitioned into $K$ and $K^\co$
and $\Sigma$ as a block matrix:
\begin{align*}
  \Sigma = \begin{pNiceMatrix}[first-row, last-col]
    K    & K^\co &   \\
    A    & B     & K \\
    B^\T & C     & K^\co
  \end{pNiceMatrix} =
  \begin{pNiceMatrix}
    I & 0 \\
    B^\T A^{-1} & I
  \end{pNiceMatrix}
  \begin{pNiceMatrix}
    A & 0 \\
    0 & C - B^\T A^{-1} B
  \end{pNiceMatrix}
  \begin{pNiceMatrix}
    I & A^{-1} B \\
    0 & I
  \end{pNiceMatrix}.
\end{align*}
Since $A = \Sigma_K$ is non-singular, this is a valid block diagonalization
of $\Sigma$. The $K^\co \times K^\co$ block in the block-diagonal matrix is
the \emph{Schur complement} of $K$ in $\Sigma$ and it is denoted by
$\Sigma^K \in \Sym_{K^\co}$. This decomposition proves Schur's formula \cite{Zhang}:
\[
  \det \Sigma = \det \Sigma_K \cdot \det \Sigma^K.
\]
It follows that every Schur complement $\Sigma^K$ in a principally regular
matrix is non-singular. Inverting both sides of the matrix equation shows
that
\begin{align}
  \label{eq:Mdual}
  (\Sigma^{-1})_{K^\co} = (\Sigma^K)^{-1}
\end{align}
and hence $\Sigma^{-1} \in \PR_N$.
Thus, matrix inversion is an involution on $\PR_N$ and the sign pattern
$\signs_{\Sigma^{-1}}$ can be computed from $\signs_{\Sigma}$ using Schur's formula:
\begin{align*}
  \signs_{\Sigma^{-1}}(K) &= \sgn \det \Sigma^{K^\co} = \sgn \det \Sigma \cdot \sgn \det \Sigma_{K^\co} \\
  &= \signs_\Sigma(N) \cdot \signs_\Sigma(K^\co).
\end{align*}
This motivates the following definitions:

\begin{definition}
Let $s$ be a sign pattern on ground set $N$. The \emph{dual} of $s$ is a
sign pattern $s^\dual$ on $N$ defined by $s^\dual(K) = s(N) \cdot s(K^\co)$.
The \emph{co-value at $K$} of a sign pattern is $s^\co(K) \defas s(N) \cdot s(K) =
s^\dual(K^\co)$.
\end{definition}

Consider subsets $K \cap L = \emptyset$ of $N$. It is clearly true that
$\Sigma_{KL} \in \PR_{KL}$ with
\[
  (\Sigma_{KL})_K = \Sigma_K.
\]
Similarly, the quotient formula for Schur complements \cite[Theorem~1.4]{Zhang}
states that $(\Sigma_{KL})^K$ is a non-singular principal submatrix of $\Sigma^K$,
namely with index set $L$, and that its Schur complement is
\[
  (\Sigma^K)^L = \Sigma^{KL}.
\]
These observations prove that principal submatrices $\Sigma_K$ and Schur
complements $\Sigma^K$ of a principally regular matrix $\Sigma$ are not
only non-singular but principally regular themselves. Moreover,
$\signs_{\Sigma_K}$ and $\signs_{\Sigma^K}$ can be computed from $\signs_{\Sigma}$.
For $\signs_{\Sigma_K}$ this is just a restriction of the function $\signs_{\Sigma}$
to the subsets of~$K$. For $\signs_{\Sigma^K}$ we use the quotient formula to obtain the
co-value $\signs_{\Sigma^K}^\co(L) = \signs_{\Sigma}^\co(KL)$.
The co-values of a sign pattern $s$ of course completely determine it
since $s^\co(L) s^\co(\emptyset) = s(L)$.

The matrix-theoretic operations of taking principal submatrices and Schur
complements furnish two combinatorial operations on sign patterns which
take a sign pattern $s$ on $N$ and produce another one on a subset
$K \subseteq N$. The patterns which arise in this way are ``minors'' or
``natural subconfigurations'' of $s$ --- one is restricting $s$ to $K$,
the other projects away from~$K$.

\begin{definition} \label{def:Minors}
Let $s$ be a sign pattern on $N$ and $K \subseteq N$. The \emph{restriction
of $s$ to $K$} is a sign pattern $s|_K$ on $K$ with values $(s|_K)(L) = s(L)$
for $L \subseteq K$. The~\emph{deletion of $K$} is $s \del K \defas s|_{K^\co}$.
The~\emph{contraction of $s$ at $K$} is the sign pattern $s \con K$ on $K^\co$
given by $(s \con K)^\co(L) = s^\co(KL)$ for $L \subseteq K^\co$. Any~sign pattern
arising from $s$ by a sequence of deletions and contractions is a \emph{minor}
of~$s$. A minor on ground set $K$ of size $k$ is called a \emph{$k$-minor}.
\end{definition}

As in matroid theory, restriction and contraction are dual to each other.
On the level of matrices, this is precisely captured by the formula
\eqref{eq:Mdual}.

\begin{lemma} \label{lemma:DelCon}
For any sign pattern $s$ on $N$ and $K \cap L = \emptyset$, we have
$(s^\dual \del K)^\dual = s \con K$.
\end{lemma}

\begin{proof}
It follows from the definitions of contraction and co-value that
$(s \con K)(L) = (s \con K)^\co(\emptyset) \cdot (s \con K)^\co(L) =
s^\co(K) \cdot s^\co(KL) = s(K) \cdot s(KL)$ for any $L \subseteq K^\co$.
On~the~other hand,
\begin{align*}
  (s^\dual \del K)^\dual(L) &= s^\dual(K^\co) \cdot s^\dual(K^\co \setminus L)
  = s^\dual(K^\co) \cdot s^\dual((KL)^\co) \\
  &= s(K) \cdot s(KL).
  \qedhere
\end{align*}
\end{proof}

The definitions above were made such that
\begin{align*}
  \signs_\Sigma^\dual &= \signs_{\Sigma^{-1}}, \\
  (\signs_\Sigma \del K^\co) = (\signs_\Sigma|_K) &= \signs_{\Sigma_K}, \\
  (\signs_\Sigma \con K) &= \signs_{\Sigma^K}.
\end{align*}
Hence, minors and duals of representable sign patterns are again representable.
We say that the class of representable sign patterns is \emph{closed} under
minors and duality.

\begin{lemma} \label{lemma:MinorClosed}
Admissible sign patterns are closed under minors and duality.
\end{lemma}

\begin{proof}
By \Cref{lemma:DelCon} it is enough to show closedness under restriction
and duality. If $s$ is admissible, then $s|_K$ evaluates to $s(\emptyset)
= \spos$ on $\emptyset$ and satisfies all the diamond axioms for subsets
of~$K$ already, proving that it is admissible. The dual $s^\dual$ evaluates
to $s(N)^2 = \spos$ on $\emptyset$.

Given a diamond $(ij|K)$ over $N$,
consider its \emph{dual} diamond $(ij|K)^\dual \defas (ij|N \setminus ijK)$.
By definition of $s^\dual$, the sign pattern $s^\dual$ satisfies the axiom
corresponding to $(ij|K)$ if and only if $s$ satisfies $(ij|K)^\dual$.
Since $s$ satisfies all diamond axioms, and hence all dual diamond axioms,
$s^\dual$ satisfies all diamond axioms and thus is admissible.
\end{proof}

\subsection{Negation and the hyperoctahedral symmetry}

We now turn to continuous symmetries of symmetric matrices and the discrete
symmetries they induce on sign patterns. Let $r$ be a non-zero real number.
Then $\PR_N$ is preserved under multiplication with $r$ and $\signs_{r\Sigma}(K)
= \sgn(r)^{|K|} \signs_{\Sigma}(K)$. Thus, only the sign of $r$ matters and this
gives an action of the group $\Set{\TT+,\TT-}$ (isomorphic to $\BB Z/2$)
on sign patterns.

\begin{definition} \label{def:Negation}
The \emph{negative} $-s$ of $s$ is given by $(-s)(K) \defas (-1)^{|K|} s(K)$.
\end{definition}

The hyperoctahedral group $\HypGrp_N$ is the Weyl group of type $B_n$ and
the symmetry group of the hypercube $C_N = [-1,1]^N$. There is an action of
\(\HypGrp_N\) on \(\PR_N\) which induces an action on principal minor vectors
and thus on sign patterns. The action is well-studied, but we describe it here
for convenience in analyzing its effect on sign patterns.

As an abstract group, $\HypGrp_N$ is the semidirect product $(\BB Z/2)^N
\rtimes \SymGrp_N$ of the group of \emph{swaps} $(\BB Z/2)^N$ and the group
of \emph{permutations}~$\SymGrp_N$. Every element in $\HypGrp_N$ can be
written as a product of a swap and a permutation. We proceed to explain
these actions on matrices and sign vectors.
The symmetric group $\SymGrp_N$ acts via orthogonal coordinate changes,
permuting rows and columns of a symmetric matrix. This induces the
corresponding permutation on principal minors and no additional sign
changes.

\begin{definition}
Let $s$ be a sign pattern on ground set $N$ and $\pi \in \SymGrp_N$.
The \emph{permuted sign pattern} $\pi \cdot s$ is defined via $(\pi \cdot s)(K)
= s(\pi(K))$. We say two sign patterns which are related by a permutation are
\emph{isomorphic}.
\end{definition}

The group of swaps is generated by reflections over coordinate hyperplanes
which leave the hypercube $C_N$ invariant. We identify elements $(z_i : i \in N)
\in (\BB Z/2)^N$ with subsets $Z \subseteq N$ via $Z = \Set{i \in N : z_i = 1}$.
To describe the action on matrices, let $Z \subseteq N$ be given and define two
diagonal $N \times N$ matrices:
\[
  A_{ii} = \begin{cases}
    1, & i \not\in Z, \\
    0, & i \in Z,
  \end{cases} \quad
  B_{ii} = \begin{cases}
    0, & i \not\in Z, \\
    -1, & i \in Z.
  \end{cases}
\]
Then the image of $\Sigma$ under the swap with $Z$ is $Z \cdot \Sigma \defas
(A - \Sigma B)^{-1}(B + \Sigma A)$. It was shown in \cite[Lemma~13]{HoltzSturmfels}
(in much greater generality) that this yields another symmetric matrix.
The principal minors of $Z \cdot \Sigma$ are computed in
\cite[Proposition~3.16]{TB} as follows:
\begin{align}
  \label{eq:HypGrpMatrix}
  \pr{K:(Z \cdot \Sigma)} = (-1)^{|Z \cap K|} \cdot \pr{Z:\Sigma}^{-1} \cdot \pr{Z \oplus K:\Sigma}.
\end{align}

\begin{remark}
This action of $\HypGrp_N$ is obtained as a quotient from a discrete subgroup
of an $\SL_2(\BB R)^N$ action on the Lagrangian Grassmannian. For more
details, see \cite[Section~3]{Geometry} and \cite[Section~3.3]{TB}.
\end{remark}

\begin{definition}
Let $s$ be a sign pattern on ground set $N$ and $Z \subseteq N$.
The \emph{swapping of $Z$} results in another sign pattern $Z \cdot s$
defined by $(Z \cdot s)(K) \defas (-1)^{|Z \cap K|} \cdot s(Z) \cdot s(Z \oplus K)$.
\end{definition}

\begin{example} \label{ex:Swap}
It is helpful to write out explicitly what it means to swap a single element~$i$:
\begin{equation}
  \label{eq:flip-action}
  (i \cdot s)(K) = \begin{cases}
    -s(i) \cdot s(K \setminus i), & i \in K , \\
     s(i) \cdot s(iK), & i \not\in K.
  \end{cases}
\end{equation}
\end{example}

\begin{remark}
One can easily check the equation
\begin{equation}
\label{eq:three-actions}
-s^\dual = N \cdot s
\end{equation}
relating the three different group actions we have defined on sign patterns.
\end{remark}

With these definitions, it is clear that the representable sign patterns are
closed under negation and the action of the hyperoctahedral group. We now show
that admissible sign patterns are closed under these symmetries as well.

\begin{lemma}
Admissible sign patterns are closed under negation and the action of $\HypGrp_N$.
\end{lemma}

\begin{proof}
The property $s(\emptyset) = \spos$ remains unchanged under both actions.
The diamond axiom is preserved under negation since the products $s(K) \cdot s(ijK)$
and $s(iK) \cdot s(jK)$ remain unchanged. A~permutation element of $\HypGrp_N$
permutes the diamonds on ground set~$N$ and therefore preserves the diamond
axioms.
All that remains is to check the action of reflections in \(\HypGrp_N\). Fix an element $z$ to swap and consider a diamond $(ij|K)$.
If $z \not\in ij$, then either all or none of the four sets of the diamond
contain~$z$. This reduces the axiom for $(ij|K)$ on $z \cdot s$ to the axiom
for $(ij|z \oplus K)$ on $s$ which holds by assumption. If $z \in ij$, then
we may assume without loss of generality that $i = z$ and compute
\begin{align*}
  (i \cdot s)(iK) \cdot (i \cdot s)(jK) &= -s(i) s(K) \cdot s(i) s(ijK) = -s(K) \cdot s(ijK), \\
  (i \cdot s)(K) \cdot (i \cdot s)(ijK) &= s(i) s(iK) \cdot -s(i) s(jK) = -s(iK) \cdot s(jK).
\end{align*}
The diamond axiom states that if the former product is negative, then the
latter product must also be negative. This translates to $s(K) = s(ijK)
\implies s(iK) = s(jK)$ which is just the contrapositive of the diamond axiom
for $(ij|K)$ as stated in \Cref{def:Admissible}.
\end{proof}

\begin{lemma} \label{lemma:PositiveHypGrp}
Every admissible sign pattern has a $\HypGrp_N$-equivalent sign pattern
in which all singletons are positively oriented.
\end{lemma}

\begin{proof}
In fact, we will prove a slightly stronger statement. Call a sign pattern
$s$ on $N$ \emph{positive definite} if $s(I) = \TT+$ for all $I \subseteq N$.
We prove:
\begin{itemize}[label=, leftmargin=3em, rightmargin=3em]
\item
If \(s\) is an admissible sign pattern on \(N = \Set{1, \ldots, n}\)
whose restriction to \(K = \Set{1, \ldots, k-1}\) is positive definite,
then all singletons in \(s\) are positively oriented or \(s\) is
$\HypGrp_N$-equivalent to a sign pattern whose restriction to
\(\Set{1, \ldots, k}\) is positive definite.
\end{itemize}
The premise of this statement for $k=1$ holds for every sign pattern.
Iterated application of it then proves the \namecref{lemma:PositiveHypGrp}.

Now we prove this stronger statement. We may freely combine permutations and
swaps to modify a given sign pattern until the above property is established.
By the hypothesis, $s|_K$ is positive definite for $K = \Set{1, \dots, k-1}$.
Assume that there is a singleton \(j\) which is not positively oriented,
so \(s(j) = \sneg\). In this case \(j\) must be at least \(k\) and after
a transposition we may assume that~$j = k$ without disturbing the positive
definiteness of $s|_K$. Based on positive definiteness and the fact
$s(k) = \TT-$, iterated application of the diamond axiom shows that
$s(kA) = \TT-$. Swapping $k$ then yields an admissible sign pattern
$k \cdot s$ such that
\begin{align*}
  (k \cdot s)(A) &= s(k) s(kA) = -s(kA) = \spos, \\
  (k \cdot s)(kA) &= -s(k) s(A) = s(A) = \spos,
\end{align*}
for all $A \subseteq K$. Therefore, \((k \cdot s)(A) = \spos\) for all
\(A \subseteq kK\).
\end{proof}

Many more combinatorial operations inspired by matrix- and matroid-theoretic
constructions can be introduced and studied; see for example~\cite[Chapter~7]{WhiteMatroids}.
This is an open-ended task and we prefer to close the section at this point
with a concrete open question. For context, recall that representable and
admissible sign patterns are minor-closed. Moreover, a sign pattern on
ground set $N$ of size $n \ge 2$ is admissible if and only if all of its
$2$-minors are admissible (on their repective $2$-element ground set).
This is because of minor-closedness (\Cref{lemma:MinorClosed}) in one
direction and in the other direction because the diamond axiom of $(ij|K)$
holds for $s$ if and only if $(ij|\emptyset)$ holds for $s \con K$
which means that $(s \con K)|_{ij}$ is admissible on ground set~$ij$. Since
the class of admissible sign patterns is minor-closed and can be completely
characterized by its $k$-minors for a finite $k$ (in this case $k=2$),
we say that it possesses a \emph{finite~forbidden~minor~description}.

\begin{question}
Does the class of representable sign patterns have a finite forbidden minor
description?
\end{question}

This question has been studied for various types of combinatorial objects.
See \cite{MSOL1,MSOL2} for a much stronger result in matroid theory,
\cite{MatusMinors} for generalizations of matroids in the context of discrete
conditional independence and \cite[Section~4.5 \& Theorem~6.13]{TB} for
Gaussians and orientable gaussoids.
In view of the arguments of \Cref{sec:Counting} (\Cref{thm:asymp-zero-prob}), there are admissible sign patterns which
are non-representable, and thus there are admissible forbidden minors for representability. In \Cref{sec:Representability} we offer \Cref{conj:minimal-forbidden-minor} as to the
unique smallest such minor, on a ground set of size \(5\).

\subsection{Actions on representations and topology}
\label{sec:ActionRep}

The symmetries from the previous section are all induced by continuous
actions on symmetric matrices. The fact that they act on their sign
patterns means that if $s$ and $s'$ are equivalent sign patterns under
any of these group actions, then $\PR_N(s)$ and $\PR_N(s')$ are homeomorphic
and in particular they have the same numbers of connected components.
For minors, we obtain a monotonicity relation.

\begin{theorem} \label{thm:ActionsTop}
Let $s$ be a representable sign pattern over $N$. If $s'$ over $N$ arises
from $s$ by negation, duality or the $\HypGrp_N$ action, then $\PR_N(s)$
is homeomorphic to $\PR_N(s')$. If $s'$ is a minor of $s$ over $K$, then
$\dim H^0(\PR_N(s)) \ge \dim H^0(\PR_K(s'))$.
\end{theorem}

\begin{proof}
The homeomorphism assertion was already proved above.
The formation of restrictions or contractions of a sign pattern corresponds
to a sequence of continuous projections of the representation space via
principal submatrices or Schur complements. Hence, we get a continuous
surjective map $\PR_N(s) \to \PR_K(s')$. The inequality follows because the
number of connected components of the image cannot exceed that of the domain.
\end{proof}

The positive diagonal matrices $P_N = \Set{ \diag(d_1, \dots, d_n) \in
\PD_N }$ give another useful action on symmetric matrices via $\Sigma
\mapsto D\,\Sigma\,D$ for $D \in P_N$. By the Cauchy--Binet formula,
\[
  \pr{I:(D\,\Sigma\,D)} = \pr{I:D}^2 \cdot \pr{I:\Sigma},
\]
and hence this action is trivial on sign patterns.
It is useful however, because it provides a strong deformation retraction
of $\PR_N(s)$ which reduces the dimension of this space without changing
its homotopy type. Below we prove a stronger version of this claim.
This dimension reduction is useful for computations.

\begin{definition}
For any sign pattern $s$, its \emph{reduced representation space} is the
subset $\hat{\PR}_N(s)$ of $\PR_N(s)$ such that all diagonal entries are
$\pm1$.
\end{definition}

\begin{proposition} \label{prop:Scaling}
If $s$ is any sign pattern, then $\PR_N(s)$ is homeomorphic to
$\hat{\PR}_N(s) \times P_N$.
\end{proposition}

\begin{proof}
Consider the map $\PR_N(s) \to \hat{\PR}_N(s) \times P_N$ sending $\Sigma
\mapsto (D\,\Sigma\,D, D)$ where $D = \diag(\sfrac1{\sqrt{|\sigma_{ii}|}} : i \in N) \in P_N$.
Since $s$ fixes the signs of $\sigma_{ii}$, this is evidently a continuous
map and its inverse given by $(\Sigma', D) \mapsto D^{-1}\, \Sigma'\, D^{-1}$
is continuous as well.
\end{proof}

\begin{remark} \label{rem:ScalingTopology}
Since $P_N$ is contractible, the orbit $P_N \cdot \Sigma$ is contractible
as well for any principally regular $\Sigma$. In particular, $\PR_N(s)$ and
$\hat{\PR}_N(s)$ have the same number of connected components.
\end{remark}

\section{Case study of \texorpdfstring{$\PR_3$}{PR\_3}}
\label{sec:PR3}

We are now ready to give a concise but exhaustive treatment of the
topological features of the smallest non-trivial case~$\PR_3$.

The image of the principal minor map $\pmm_3: \Sym_3 \to \BB R^{2^3}$ is a
semialgebraic set in $\BB R^8$ whose coordinates we denote by $p_K$,
$K \subseteq \Set{1,2,3}$. It is entirely contained in the hyperplane
$p_\emptyset = 1$ and inside of this $\BB R^7$, it is cut out by the
(dehomogenized) \emph{$2 \times 2 \times 2$ hyperdeterminant} \cite{HoltzSturmfels}:
\begin{align*}
  h &\defas p_{123}^2 - 2\, p_{123}\, p_{13}\, p_{2} + p_{13}^2\, p_{2}^2 - 2\, p_{1}\, p_{123}\, p_{23} + 4 p_{12}\, p_{13}\, p_{23} - 2\, p_{1}\, p_{13}\, p_{2}\, p_{23} + {} \\
    &\hphantom{{}={}} p_{1}^2\, p_{23}^2 - 2\, p_{12}\, p_{123}\, p_{3} + 4\, p_{1}\, p_{123}\, p_{2}\, p_{3} - 2\, p_{12}\, p_{13}\, p_{2}\, p_{3} - 2\, p_{1}\, p_{12}\, p_{23}\, p_{3} + p_{12}^2\, p_{3}^2.
\end{align*}
together with the inequalities $p_{ij} \le p_i \, p_j$ for all $i \neq j$.
The real hypersurface $h = 0$ intersects precisely $128 - 24 = 104$ open
orthants in $\BB R^7$; the $24$ orthants which it misses correspond to
non-admissible sign patterns on $\Set{1,2,3}$. Adding the inequality constraints
$p_{ij} \le p_i \, p_j$ completes the description of the image of $\pmm_3$.
The number of orthants that it intersects drops to $38$ which correspond
exactly to the admissible sign patterns; this proves that all of them are
representable.

The hyperdeterminant is a polynomial of degree $4$ in $7$ variables with
$12$ terms. For a symbolic algorithm like cylindrical algebraic decomposition
which can compute the representable sign patterns from the definition
of $\pmm_3$, these numbers are already quite high (and the situation for $n=4$
is much worse). To classify sign patterns according to representability,
a better approach is to work in the coordinates of $\PR_N$ and to exploit
the various symmetries described in \Cref{sec:Constructions}.

\begin{table}
\begin{tabular}{cccc}
$\HypGrp_3$ representative & Representation & Orbit size & Connected components \\ \hline
\TT{++++++++} & $(0, 0, 0)$    & $8$  & $1$  \\
\TT{+++++++-} & $(-2, -2, -1)$ & $8$  & $4$  \\
\TT{++++++--} & $(-2, -2, -4)$ & $12$ & $2$  \\
\TT{+++++---} & $(-2, -4, -4)$ & $8$  & $4$  \\
\TT{++++----} & $(-4, -4, -4)$ & $2$  & $16$
\end{tabular}
\caption{Sign patterns grouped by hyperoctahedral orbits. The sign vector
$s$ is given as a string indexed by $\emptyset, 1, 2, 3, 12, 13, 23, 123$
in order.
The representation only lists off-diagonal entries $(\sigma_{12}, \sigma_{13},
\sigma_{23})$; the diagonal entries are all equal to~$3$. In~total there
are $38$ sign patterns and $128$ connected components.}
\label{tab:B3}
\end{table}

\pagebreak %

The $38$ admissible sign patterns on $n=3$ split into only $5$ orbits modulo
$\HypGrp_3$. Since representability is preserved by this group action,
the classification task is reduced to only $5$~instances which are listed
in the first column in \Cref{tab:B3}. Furthermore, by \Cref{lemma:PositiveHypGrp}
we may pick a representative~$s$ of the orbit on which all $1 \times 1$-minors
are positive. Lastly \Cref{prop:Scaling} allows us to fix the diagonals to~$1$.
In this setting, the inequalities $\pr{i:\Sigma} > 0$ are trivial. The
$2 \times 2$-minor inequality $s(ij) \cdot \pr{ij:\Sigma} > 0$ reduces
to whether $\sigma_{ij}$ lies inside the interval $(-1,1)$ or inside
the union $(-\infty,-1) \cup (1,\infty)$. The last inequality
$s(123) \cdot \pr{123:\Sigma} > 0$ has degree~$3$ in $3$ variables
and $5$~terms which is simpler in every metric than the hyperdeterminant.
These systems are solved instantly in \TT{Mathematica}. The entries of
particularly simple points in $\PR_3(s)$ are given in the second column
of \Cref{tab:B3}.

\begin{figure}
\centering
\begin{minipage}{.3\linewidth}
\centering
\includegraphics[width=\linewidth]{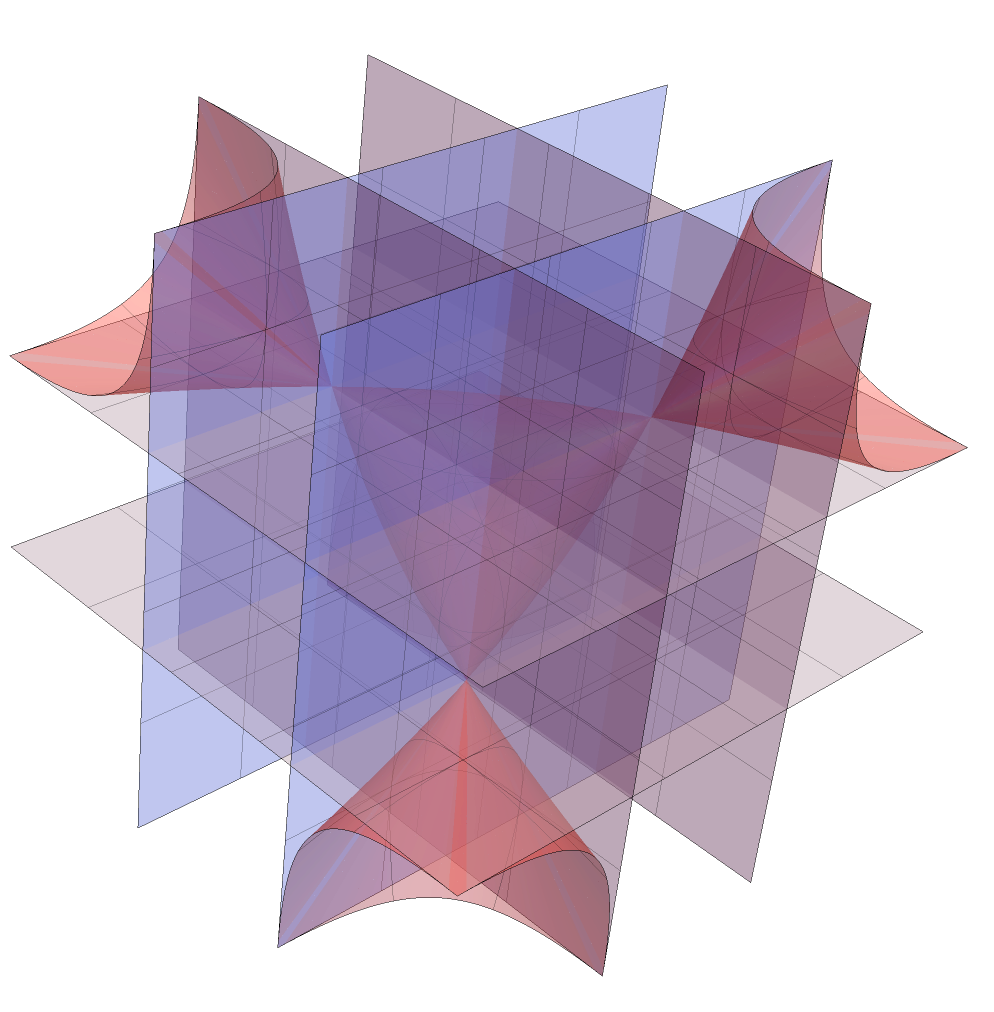}
\end{minipage}
~
\begin{minipage}{.3\linewidth}
\centering
\includegraphics[width=\linewidth]{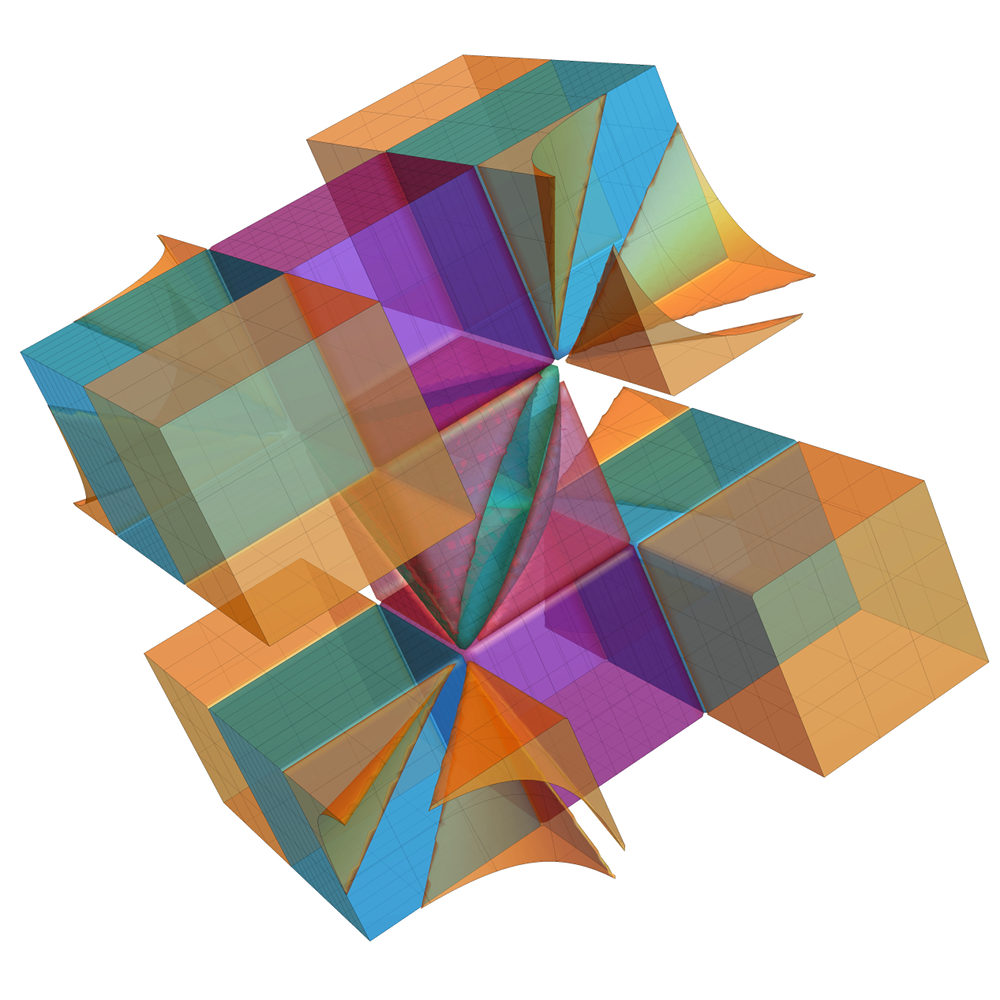}
\end{minipage}
~
\begin{minipage}{.3\linewidth}
\centering
\includegraphics[width=.3\linewidth]{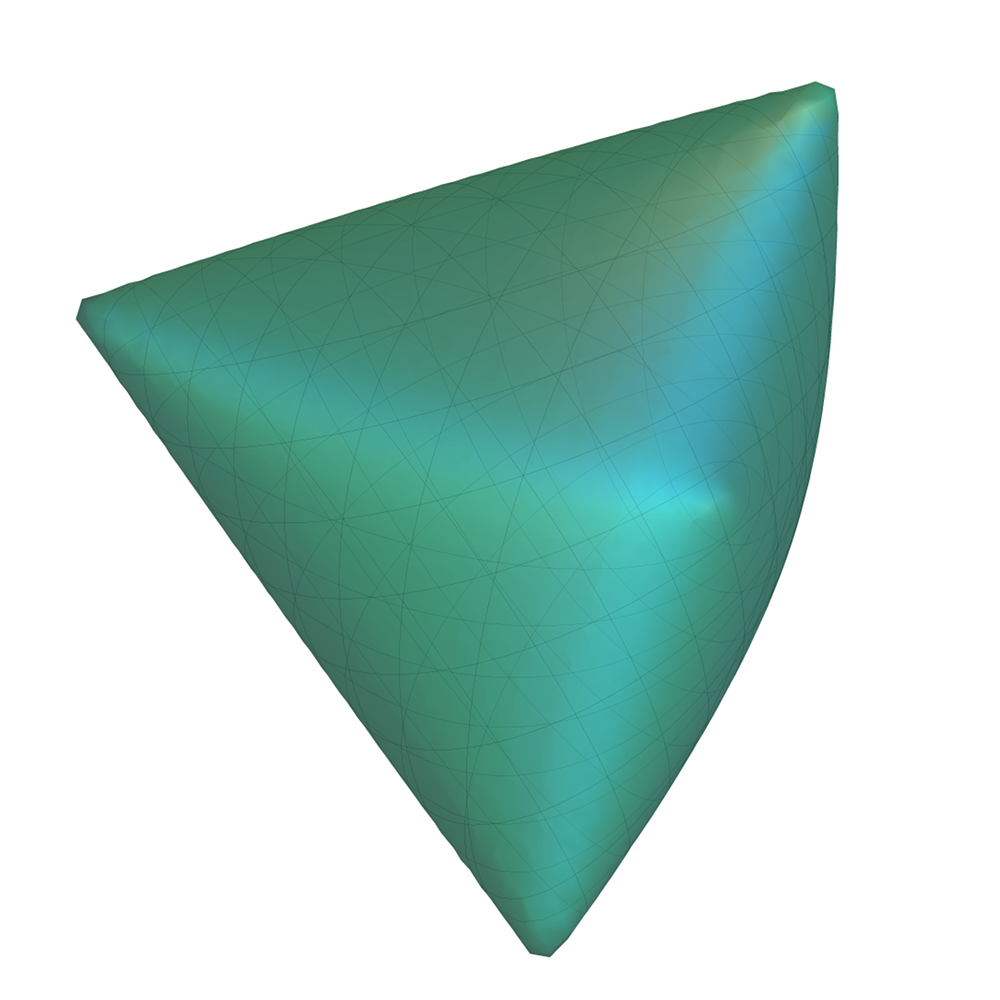}
\includegraphics[width=.3\linewidth]{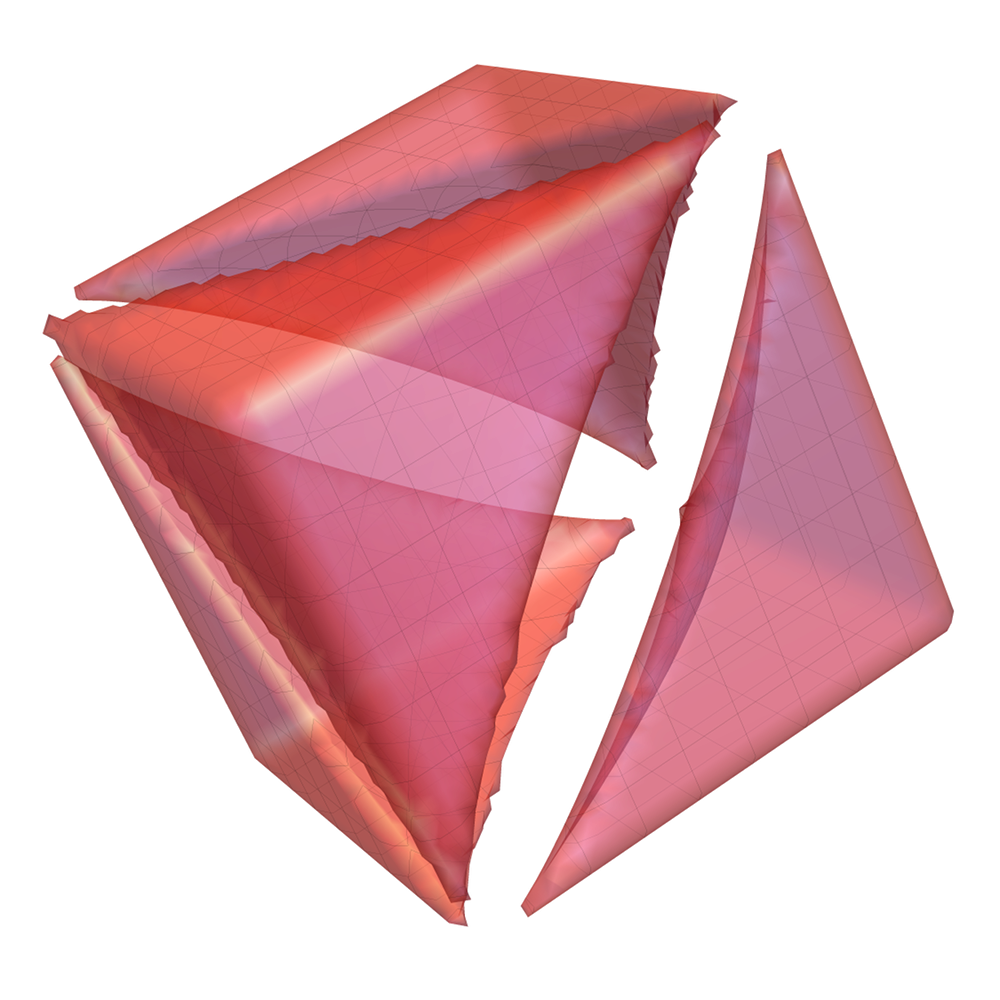} \\
\includegraphics[width=.3\linewidth]{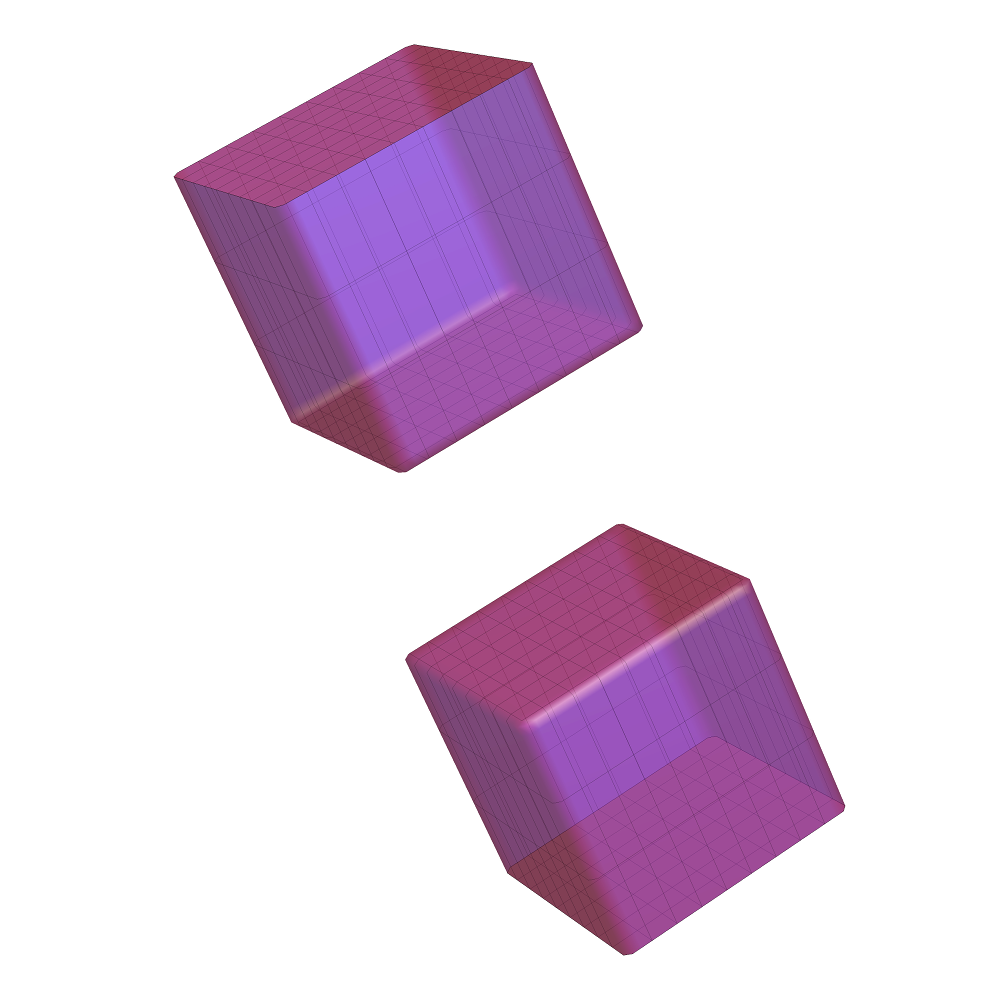}
\includegraphics[width=.3\linewidth]{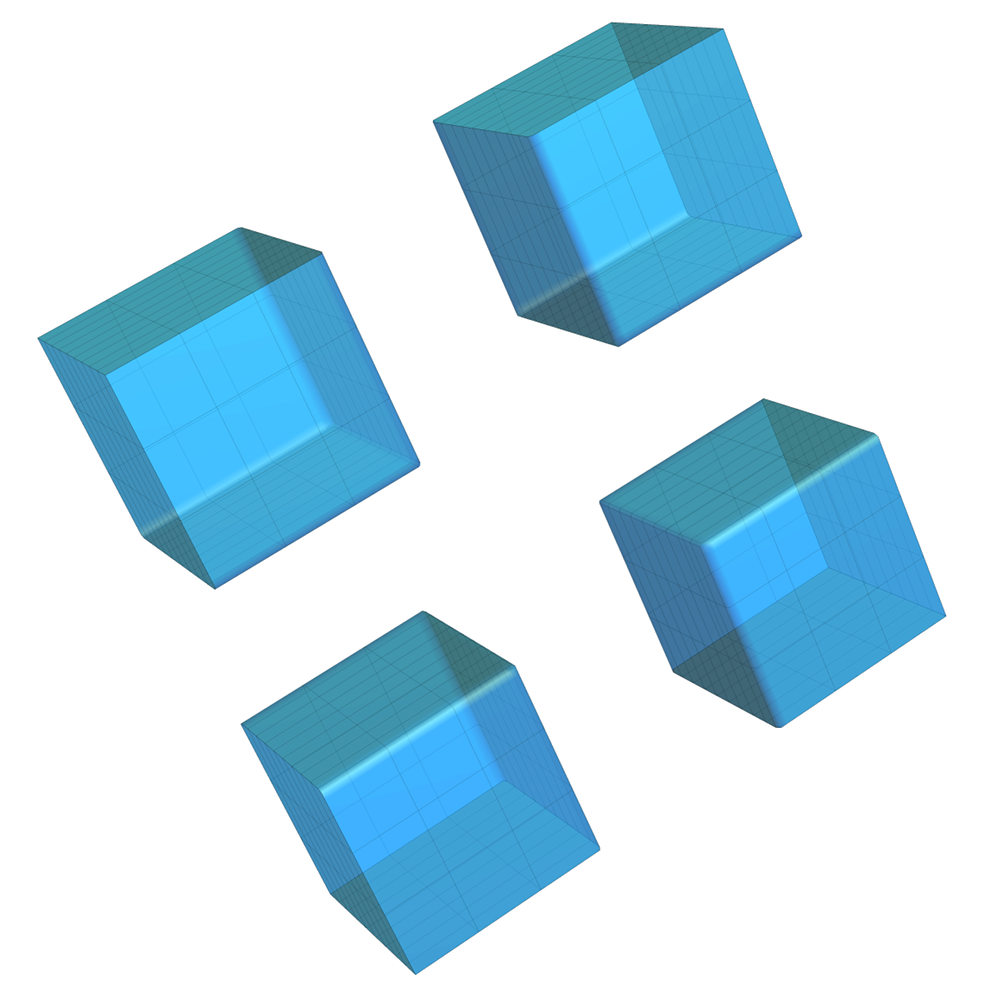}
\includegraphics[width=.35\linewidth]{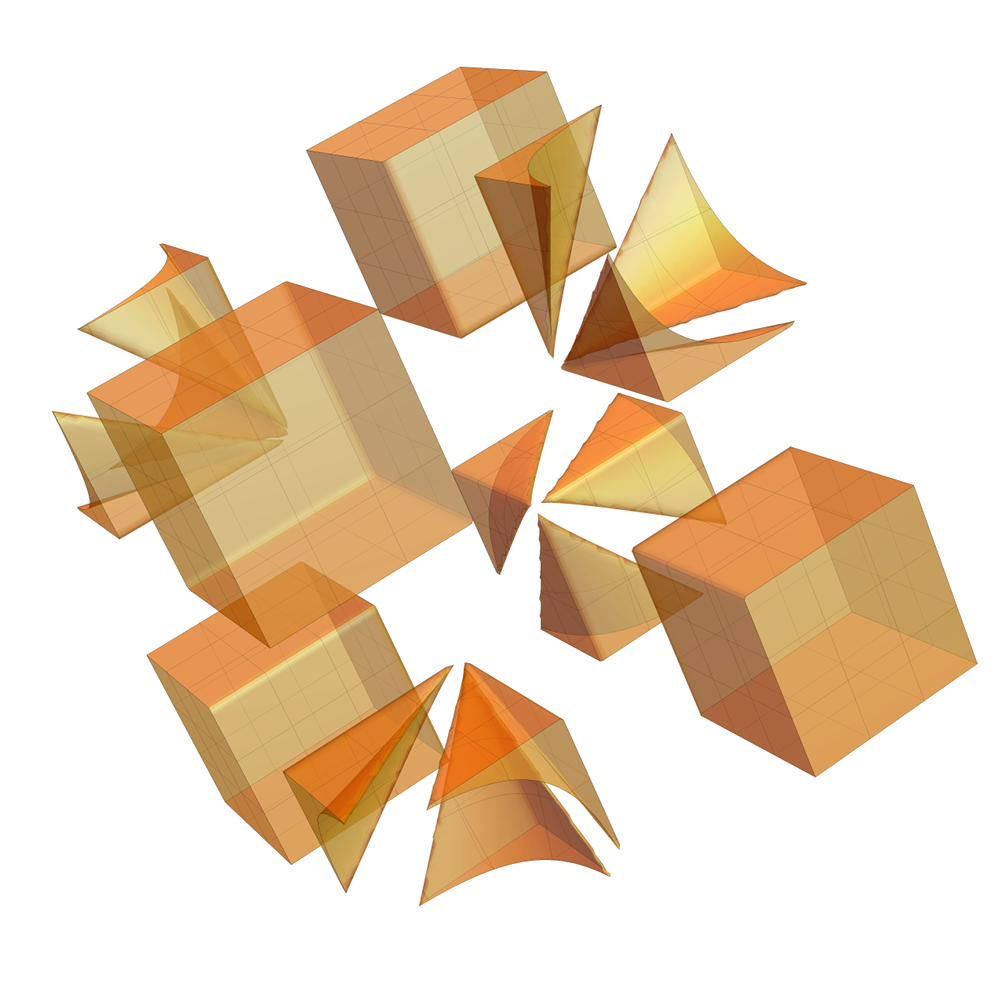}
\end{minipage}
\caption{The arrangement of hypersurfaces $\protect\pr{K:\Sigma} = 0$
whose complement is $\hat{\PR}_3$ as well as the pieces $\hat{\PR}_3(s)$
for the $\HypGrp_3$ representatives $s$ from \Cref{tab:B3}, together
in the middle and and separately (on different scales) on the right.}
\label{fig:Reps}
\end{figure}

The reduced representation spaces $\hat{\PR}_3(s)$ are $3$-dimensional and
plotted in~\Cref{fig:Reps}. Their numbers of connected components can
intuitively be seen from these pictures and they are recorded in the last
column in \Cref{tab:B3}. They match the numbers obtained by symbolic
computations using cylindrical algebraic decomposition in \TT{Mathematica}.
By the remarks in \Cref{sec:ActionRep}, the number of connected components
of $\PR_3$ can be computed by taking the dot product of the ``Orbit size''
and ``Connected components'' columns of this table.
This number is $128 = 2^7$. However, the equality with the total number of sign
patterns $s \in \Set{\TT+, \TT-}^{2^3}$ (satisfying $s(\emptyset) = \TT+$) is
only accidental as we know that not every sign pattern is representable and
that representation spaces need not be connected. We pose the following
computational problem which asks whether the connected components of $\PR_3$
are still ``uniformly basic semialgebraic''.

\begin{challenge}
If possible, find 7 semialgebraic functions such that the $128 = 2^7$ regions
defined by placing sign constraints on them are precisely the connected
components of~$\PR_3$. Can these functions be chosen polynomial?
\end{challenge}

By the asymptotic result in \Cref{cor:pr-h0}, we know that in general
the number of connected components of $\PR_n$ cannot be $2^{2^n - 1}$
(as is the case for $n=3$). We do not have enough data to conjecture
any other closed form expression. Even the number of connected
components of $\PR_4$ is not known.

\begin{challenge}
Find a CAD of $\PR_4$ inside $\Sym_4 \cong \BB R^{10}$ (or of $\hat{\PR}_4$
in $\BB R^6$) or of its image under the principal minor map inside $\BB R^{16}$.
How many connected components does $\PR_4$ have?
\end{challenge}

\begin{remark}
We show in \Cref{sec:Rep5} below that all admissible sign patterns on $n=4$
are representable. Combining the minor inequality from \Cref{thm:ActionsTop}
and the data on representation spaces of their $3$-minors in \Cref{tab:B3},
one obtains a lower bound of $7\,848$ for the number of connected components
of~$\PR_4$.
\end{remark}

\begin{remark}
Joseph Cummings reports in private communication that the algorithm
described in \cite{SemialgebraicConnectivity,SmoothConnectivity} yields
an upper bound of $24\,352$ connected components for $\PR_4$, provided
that the monodromy loops exhausted all critical points of the associated
rational function. However, the amount of critical points makes
certification procedures impractical. The same number has been obtained
using the \TT{HypersurfaceRegions} package developed independently by
Breiding, Sturmfels and Wang \cite{HypersurfaceArrangements}. The source
code and more details about this computation are available on our website.
\end{remark}

\section{Enumeration and asymptotics} \label{sec:Counting}

In this section, we study the sequence \(a_n\) counting the number of admissible
sign patterns of size \(n\), and the sequence \(r_n\) counting the number of
sign patterns of principally regular real symmetric \(n \times n\) matrices.
There are two clear inequalities from the definitions:
\[
  r_n \le a_n \le 2^{2^n - 1}
\]
since representable sign patterns are admissible and admissible sign patterns
satisfy $s(\emptyset) = \TT+$.
The main result of the section is \Cref{thm:asymp-zero-prob}, which states that
\(a_n\) grows much quicker than \(r_n\), but we give the best bounds
we can for the asymptotics of both sequences. We begin with explicit
computations on small ground sets $n \le 5$.

\subsection{Enumerating sign patterns}
\label{sec:counting-data}

The definition of admissible sign patterns on ground set $N$ is a simple
boolean formula in the $2^n$ variables $V_K$, one for each $K \subseteq N$,
mapping $s(K) = \spos$ to $V_K = \TT{true}$. The resulting boolean formula
in conjunctive normal form (CNF) has $1 + \binom{n}{2} 2^n$ clauses
(four clauses for each diamond). Checking whether a boolean formula in
CNF has a satisfying assignment is a well-known problem in computer
science called \TT{SAT}. Even though \TT{SAT} is an \TT{NP}-complete
problem, state-of-the-art solvers are extremely efficient pieces of
software and are able to \begin{inparaenum}
\item produce a satisfying assignment (or a proof of unsatisfiability),
\item count and
\item enumerate all solutions
\end{inparaenum}
in cases of practical interest.

\begin{table}
\begin{tabular}{cccc}
\(n\) & \(a_n\) & \(a'_n\) & \(r'_n\) \\ \hline
$2$        & $6$ & $2$ & \(2\) \\
$3$        & $38$ & $5$ & \(5\) \\
$4$        & $990$ & $24$ & \(24\) \\
$5$        & $395\,094$ & $434$ & \(\geq 433\) \\
$6$        & $33\,433\,683\,534$ & $\le 7\,109\,686\,748$ &
\end{tabular}
\caption{The number of admissible sign patterns \(a_n\), the number of hyperoctahedral
orbits \(a'_n\), and the number of representable hyperoctahedral orbits \(r'_n\)
for ground sets of sizes $n = 2, \dots, 6$. The~upper bound on $a'_6$ uses additional
axioms implied by \Cref{lemma:PositiveHypGrp}.}
\label{tab:Signs}
\end{table}

We have employed the three SAT solvers \TT{GANAK} \citesoft{GANAK},
\TT{sharpSAT-TD} \citesoft{sharpsatTD} and \TT{nbc\_minisat\_all} \citesoft{TodaSAT}
through the \TT{CInet::ManySAT} \citesoft{ManySAT} interface
to count and enumerate admissible sign patterns on up to $n=5$
and subsequently reduced the listing modulo the hyperoctahedral group.
Our~results are summarized in the first three columns of \Cref{tab:Signs}.
These computations led to the creation of sequence $a_n = \text{%
\href{https://oeis.org/A375346}{\TT{A375346}}}$ in the On-Line Encyclopedia
of Integer Sequences~\cite{OEIS}.

For $n=6$, the number of admissible sign patterns is too large to be enumerated
and symmetry-reduced. We instead offer the number of admissible sign patterns
all of whose singletons are positively oriented. By~\Cref{lemma:PositiveHypGrp}
this gives an upper bound on the number of hyperoctahedral orbits which appears
to be quite weak. The predicted upper bound for $a_6'$ is still too large to
enumerate these sign patterns and perform a symmetry reduction modulo~$\HypGrp_6$.

\begin{challenge}
Find hyperoctahedral representatives for the admissible sign patterns on $n=6$.
\end{challenge}

\subsection{Asymptotics of admissible sign patterns}
\label{sec:admissible-asymp}

The number of admissible sign patterns is doubly exponential. The following
bounds on its logarithm are easy to derive:
\begin{equation}
  \label{eq:Bounds}
  2^{n-1} \le \log_2 a_n \le 1 + 2 \log_2 a_{n-1}.
\end{equation}
The lower bound comes from sign patterns satisfying $s(K) =
(-1)^{|K|/2}$ whenever $|K|$ is even. Any such $s$ is vacuously
admissible, no matter the values on sets with odd cardinality. This
gives $\sum_i \binom{n}{i}$ patterns, where the sum extends over all
$1 \le i \le n$ with $i$ odd, summing to a total of~$2^{n-1}$.
The upper bound follows from the map $A_{mN} \to \Set{\spos,\sneg} \times
A_N \times A_N$, $s \mapsto (s(mN),\, s \del m, s \con m)$, where
$A_N$ denotes the set of admissible sign patterns on ground set~$N$.
That this map is injective is an easy consequence of \Cref{def:Minors}.
Asymptotically, this bounds $\log_2 a_n$ between $2^{n-1}$~and~$2^n$.
The finer details are settled by the following result which was
suggested to us by Andrey Zabolotskiy together with references to
two papers in which the same sequence appears, counting
\begin{inparaenum}
\item vertices in a certain class of Lipschitz polytopes \cite{Wasserstein},
and \item hypercube Adinkra height functions \cite{Heights}.
\end{inparaenum}

\begin{theorem} \label{thm:Asymptotic}
The number of admissible sign patterns satisfies $\log_2 a_n = 2^{n-1} + \CC O(1)$.
\end{theorem}

\begin{proof}
Let $C_N$ denote the \emph{hypercube graph} whose vertices are the subsets
of~$N$ with an edge between $I, J \subseteq N$ if and only if $J = Ij$ for
some $j \not\in I$. The main result of \cite{HomHamming} is that the number
of $3$-colorings of $C_N$ which assign a fixed color to $\emptyset$ is
$(2e \pm e^{-\Omega(n)}) 2^{2^{n-1}}$. We will establish a bijection between
admissible sign patterns and these $3$-colorings, which implies~the~claim.

Identify the colors with the additive group $\BB Z/3\BB Z$ and let $F_N$ be
the set of all proper $3$-colorings $c: 2^N \to \BB Z/3\BB Z$ of $C_N$ such
that~$c(\emptyset) = 0$; let $A_N$ be the set of admissible sign patterns.
Given $c \in F_N$, define a sign pattern $s$ by setting $s(\emptyset) \defas \spos$
and, continuing recursively,
\begin{align}
  \label{eq:defs}
  s(iK) \defas \begin{cases}
     s(K), & \text{if $c(iK) = c(K) + 1$}, \\
    -s(K), & \text{if $c(iK) = c(K) - 1$}.
  \end{cases}
\end{align}
Since there is an edge between $K$ and $iK$ in $C_N$ and hence $c(iK) \not= c(K)$,
this imposes at least one value on each subset of~$N$. We show by induction on
$|K|$ that $s$ is well-defined and admissible. These facts are clear for
$0 \le |K| \le 1$ since $s(\emptyset) = \spos$ is fixed a~priori and there
is only a single edge between $\emptyset$ and $i$ for each $i \in N$.

Now assume that the partial sign pattern $s$ defined on sets of $\text{cardinalities} \le k+1$
is well-defined and admissible. Take any set $L$ of cardinality $k+2$. To show
that the value $s(L)$ does not depend on the edge chosen in \eqref{eq:defs},
write $L = ijK$ and consider the two edges from $iK$ and $jK$ to $ijK$.
Suppose that $s(ijK)$ is not well-defined. Up to symmetries, there are two
cases to consider:
\begin{enumerate}[label=(\alph*)]
\item Assume that \eqref{eq:defs} mandates $s(ijK) = s(iK)$ as well as
  $s(ijK) = s(jK)$ but that $s(iK) = -s(jK)$. By definition of $s$,
  this means $c(ijK) = c(iK) + 1 = c(jK) + 1$, hence $c(iK) = c(jK)$.
  Since $iK$ and $jK$ have the same color and an edge to the same set~$K$,
  the induction hypothesis implies $s(iK) = s(jK)$, so this cannot occur.

\item Assume that \eqref{eq:defs} mandates $s(ijK) = s(iK)$ as well as
  $s(ijK) = -s(jK)$ but that $s(iK) = s(jK)$. The definition of $s$
  implies $c(ijK) = c(iK) + 1 = c(jK) - 1$, so $c(iK) \not= c(jK)$.
  Having different colors and the same neighbor~$K$, this implies
  $s(iK) \not= s(jK)$, another contradiction.
\end{enumerate}
This shows that $s(ijK)$ is well-defined. To show admissibility, assume that
$s(iK) \not= s(jK)$. This implies $c(iK) \not= c(jK)$ as seen above. Since
$K$, $iK$, $jK$ and $ijK$ form a diamond, the assumption that $c$ is a
$3$-coloring implies $c(ijK) = c(K)$ and hence $s(ijK) = s(iK) = -s(K)$
or $s(ijK) = -s(iK) = -s(K)$ depending on whether $c(iK) = c(K) \pm 1$.
Thus, $s$ is admissible.

The inverse of this map sends an $s \in A_N$ to a $c \in F_N$ via a
similar construction: $c(\emptyset) \defas 0$ and %
\[
  c(iK) \defas \begin{cases}
    c(K) + 1, & \text{if $s(iK) = s(K)$}, \\
    c(K) - 1, & \text{if $s(iK) = -s(K)$}.
  \end{cases}
\]
The well-definedness of $c$ follows via a case distinction analogous to
the one above. As it is well-defined, the $3$-coloring property is obvious.
\end{proof}

\subsection{Asymptotics of representable sign patterns}
\label{sec:representable-complexity}
Our upper bound on the number of representable sign patterns comes from bounding
the algebraic complexity of the minor equations which cut \(\PR_N\) out of \(\Sym_N\).
First we quote a result of Basu, Pollak, and Roy
in a simplified form for our case. For any family of polynomials
$\CC P = \Set{P_1, \dots, P_\ell} \subseteq \BB R[X_1, \dots, X_k]$ we
consider the space $\CC C(\CC P) \defas \BB R^k \setminus \bigcup_{i=1}^{\ell} V(P_i)$
which is the complement of the arrangement of varieties~$P_i = 0$.

\begin{proposition}[{\cite[Theorem 1]{B.P.R1996}}]
  \label{prop:complexity-bound}
  For any family of \(\ell\) polynomials \(\mathcal{P} = \Set{P_1, \ldots, P_\ell}
  \subseteq \mathbb{R}[X_1, \ldots, X_k]\), where each polynomial has degree
  at most \(d\), we have
  \[
    \dim H^0(\CC C(\CC P)) \le \binom{\ell}{k} \CC O(d)^k.
  \]
\end{proposition}

\begin{corollary}
  \label{cor:pr-h0}
  The number of connected components of \(\PR_N\) is bounded by
  \[
    \log_2(\dim H^0(\PR_N)) \leq \CC O(n^3).
  \]
\end{corollary}

\begin{proof}
A real symmetric matrix is given by \(\binom{n+1}{2}\) coordinates. For
$\CC P = \Set{ \pr{I:\Sigma} : I \subseteq N }$, we evidently have
$\CC C(\CC P) = \PR_N$. The family $\CC P$ contains $2^n$ polynomials of
degree at most~$n$. Plugging into \Cref{prop:complexity-bound} gives
\[
  \dim H^0(\PR_N) \leq \binom{2^n}{\binom{n+1}{2}} \CC O(n)^{\binom{n+1}{2}}.
\]
Using the upper bound $\binom{x}{a} \leq \frac{1}{a!}x^a$
yields $\log_2(\dim H^0(\PR_N)) \leq \CC O(n^3)$.
\end{proof}

The corollary above bounds the sequence \(r_n\) beneath a singly exponential
function in \(n\), while our lower bound on the sequence \(a_n\) is doubly
exponential.

\begin{proposition}
  \label{prop:representable-upper-complexity}
  For all $n$: $\log_2(r_n) \leq \CC O(n^3)$.
\end{proposition}

\begin{proof}
  The map $\signs: \PR_N \to \{\spos, \sneg\}^{2^N}$
  is a continuous map to a discrete space, and so the cardinality of the image
  is bounded by the number of connected components of the domain, which is
  bounded by \Cref{cor:pr-h0}.
\end{proof}

\begin{remark}
In~\cite{NelsonBound}, Nelson shows that almost all matroids are non-representable
over \emph{any field} by combining a doubly exponential lower bound for matroids
with a newly derived singly exponential upper bound on representable matroids.
Interestingly, his bound on the logarithm is $\CC O(n^3)$ and matches our bound
from \Cref{prop:representable-upper-complexity}. The same dichotomy of single
vs.\ double exponential growth also holds in the realm of gaussoids with the
same $\CC O(n^3)$ bound; cf.~\cite{ConstructionMethods}.
For combinatorial objects capturing an ``orientation'' such as our admissible sign
patterns, the ground field needs to be ordered and by Tarski's transfer principle
\cite[Section~11.2]{Marshall} only representability over the real numbers needs
to be considered to establish an upper bound. In this case, \cite{B.P.R1996} is
a very versatile tool. It makes no problem to deduce from it analogous $\CC O(n^3)$
upper bounds on the numbers of representable uniform oriented matroids and gaussoids.
\end{remark}

Combining \Cref{thm:Asymptotic} and \Cref{prop:representable-upper-complexity}
we get our main result:

\begin{theorem}
  \label{thm:asymp-zero-prob}
  Let \(p_n\) be the probability that a uniformly chosen admissible sign pattern
  of size \(n\) is representable. Then \(\lim_{n \to \infty} p_n = 0\) and in fact
  \[
    -\log_2(p_n) \geq \Omega(2^n).
  \]
\end{theorem}

The asymptotic distinction between the admissible sign patterns
and the representable ones appears even when we only consider size \(3\) minors.

\begin{theorem}
  \label{thm:asymp-zero-prob-deg3}
  Let \(\tilde p_n\) be the probability that a uniformly chosen assignment of signs to
  the size 3 subsets of \(\{1, \ldots, n\}\) occurs in a representable sign pattern. Then
  \(\lim_{n \to \infty} \tilde p_n = 0\).
\end{theorem}
\begin{proof}
  Let \(c_n\) be the number of connected components of $\CC C(\tilde{\CC P})$
  for $\tilde{\CC P} = \Set{ \pr{I:\Sigma} : I \subseteq N,\, |I| = 3 }$.
  There are only \(\binom{n+1}{3}\) such minors, and so again using
  \Cref{prop:complexity-bound}, we obtain
  \[
    c_n \leq \binom{\binom{n+1}{3}}{\binom{n+1}{2}} c_0^{\binom{n+1}{2}},
  \]
  for some constant $c_0$. With the same approximations as in \Cref{cor:pr-h0}
  we obtain
  \[
    \log c_n \leq \CC O(n^2\log(n))
  \]
  but the number of assignments of signs to size 3 subsets of \(\{1, \ldots, n\}\) is \(2^{\Omega(n^3)}\).
\end{proof}

We note that every assignment of signs to subsets of size $3$ occurs
in an \emph{admissible} sign pattern, by the construction of the lower
bound in \eqref{eq:Bounds}.

The best lower bound for \(r_n\) that we know of is as follows.
\begin{remark}
  \label{rem:representable_lower_construction}
  Given a vector of prescribed signs \((s(K))_{K: |K| = 2}\) for the
  size \(2\) principal minors, the matrix \((c_{ij})\)
  defined by
  \[
    c_{ij} =
    \begin{cases}
      1 & \text{ if }i = j \\
      2 & \text{ if }s(ij) = \sneg \\
      0 & \text{ otherwise}
    \end{cases}
  \]
  has principal minors of those signs by construction, and a small perturbation
  of the entries will give a principally regular matrix. Thus,
  $\log_2(r_n) \geq \Omega(n^2)$.
\end{remark}

Finally, we make a remark about the number of equivalence classes under the
hyperoctahedral group. Let \(a'_n\) be the number of equivalence classes of
size \(n\) sign patterns under the hyperoctahedral group action, and \(r'_n\)
the number of representable equivalence classes.

\begin{remark}
Since $|\HypGrp_n| = 2^n n!$, we have $\log_2 |\HypGrp_n| \le \CC O(n \log n)$,
and both \(\log_2(a_n)\) and \(\log_2(r_n)\) are at least \(\Omega(n^2)\),
the exponential growth rates are unaffected by quotienting by the group action.
That is to say, \(\log_2(a'_n) = \Theta(\log_2(a_n))\) and
\(\log_2(r'_n) = \Theta(\log_2(r_n))\).
\end{remark}

\section{Representability and topology}
\label{sec:Representability}

We now turn to questions about representability and the topology of the
spaces of representations. As in \Cref{sec:PR3} we work in the affine space
$\Sym_N$ where a sign pattern $s$ gives rise to the semialgebraic set~$\PR_N(s)$
in $\Sym_N$. This set is defined by the strict determinantal inequalities
$s(K) \cdot \pr{K:\Sigma} > 0$ for all $K \subseteq N$, and $s$ is
representable if and only if $\PR_N(s)$ (or, equivalently, $\hat{\PR}_N(s)$)
is non-empty. Since these sets are open, the existence of a real point in
them implies existence of a rational point. On our website we provide a
database of exact rational points certifying all representability claims
made here.

\subsection{Representability for \texorpdfstring{$n \le 5$}{n <= 5}}
\label{sec:Rep5}

To check if a sign pattern is representable, one may naïvely generate
random matrices with rational entries having small numerators and
denominators (and fixed diagonal entries via \Cref{prop:Scaling}).
This strategy, which we implemented in \TT{sagemath} \citesoft{sagemath},
is surprisingly effective: it finds simple rational representations
for all hyperoctahedral orbits on $n \le 4$ with little effort.
For $n=5$, it yields witnesses for the representability of all but one
of the $434$ $\HypGrp_5$-orbits. These results are summarized in column
four of \Cref{tab:Signs}. We conjecture that the last remaining equivalence
class is non-representable. This final class is represented by the
sign pattern $s_*$ shown in \Cref{fig:last-pattern}.

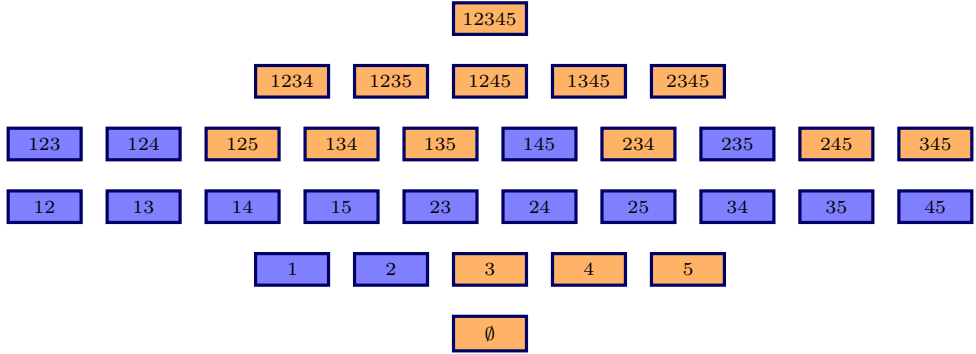
\begin{figure}[ht]
  \centering
  \begin{tikzpicture}
    \tikzset{box/.style={font=\tiny, draw=blue!40!black, very thick, inner sep=1.3mm, text width=7mm, align=center}}
    \tikzset{spos/.style={box, fill=orange!60!white}}
    \tikzset{sneg/.style={box, fill=blue!50!white}}
    \tikzset{every matrix/.style={column sep=3mm}}

    \matrix (m5) [matrix of nodes]{ |[spos]| 12345 \\ };
    \matrix (m4) [matrix of nodes, below=1mm of m5]{ |[spos]| 1234 & |[spos]| 1235 & |[spos]| 1245 & |[spos]| 1345 & |[spos]| 2345 \\ };
    \matrix (m3) [matrix of nodes, below=1mm of m4]{ |[sneg]| 123 & |[sneg]| 124 & |[spos]| 125 & |[spos]| 134 & |[spos]| 135 & |[sneg]| 145 & |[spos]| 234 & |[sneg]| 235 & |[spos]| 245 & |[spos]| 345 \\ };
    \matrix (m2) [matrix of nodes, below=1mm of m3]{ |[sneg]| 12 & |[sneg]| 13 & |[sneg]| 14 & |[sneg]| 15 & |[sneg]| 23 & |[sneg]| 24 & |[sneg]| 25 & |[sneg]| 34 & |[sneg]| 35 & |[sneg]| 45 \\ };
    \matrix (m1) [matrix of nodes, below=1mm of m2]{ |[sneg]| 1 & |[sneg]| 2 & |[spos]| 3 & |[spos]| 4 & |[spos]| 5 \\ };
    \matrix (m0) [matrix of nodes, below=1mm of m1]{ |[spos]| $\emptyset$ \\ };
  \end{tikzpicture}
  \caption{An admissible sign pattern \(s_*\) of size 5 for which we have no
    representation, depicted as a coloring of the lattice of subsets of
    \(\{1, \ldots, 5\}\). Positive signs are in orange, and negative in blue.}
  \label{fig:last-pattern}
\end{figure}

We note that $s_*$ is uniquely determined by its values on subsets of odd
order and the assumption that it is admissible. Indeed, the signs assigned
to singletons and \eqref{eq:diamond} imply the negativity of $s_*(ij)$,
with $i \in 12$ and $j \in 345$. The negativity of $s_*(12)$ follows from
the values at $13$, $123$ and~$1$. Similar proofs can be found for the
negativity of $s_*(34)$, $s_*(35)$ and $s_*(45)$ showing that the values
of all principal minors of order two are implied by the values at orders
one and three via~\eqref{eq:diamond}. The positivity of $s_*(1234)$
follows directly from the values at $123$, $234$ and $23$; and~similar
proofs can be found for the remaining $4$-minors.
As in \Cref{sec:PR3}, for representability testing it suffices to consider
the restricted space $\hat{\PR}_5(s_*)$ in which the diagonals are fixed
to~$\pm1$. The~previous discussion shows that $\hat{\PR}_5(s_*)$ can
equivalently be described by a polynomial system in 10~variables with
only 10~sign constraints on order three principal minors and the positivity
of the determinant.

Despite these reductions, we can neither find a representation computationally
nor prove that there is none. In the search for a representation, we have
employed the polynomial programming solver from the SCIP suite \citesoft{SCIP}.
Unfortunately, during the solving of this problem (in its most straightforward
formulation), a matrix becomes numerically singular and SCIP aborts.
On the other hand, cylindrical algebraic decomposition (CAD) is an
exact symbolic algorithm which, given a semialgebraic set, either finds a
point inside or conclusively determines that it is empty. We have used the
implementation in \TT{Mathematica} \citesoft{Mathematica} via its
\TT{FindInstance} function but could not find a formulation on which the
algorithm terminates.

We can offer one more insight. The sign pattern $s_*'$ obtained from $s_*$
by flipping the sign of $12345$ is still admissible and inequivalent to
$s_*$ under the hyperoctahedral group. It is therefore representable
according to our computations. A rational representation is given by
\[
  \Sigma_*' = \begin{pmatrix}
    -1 & \sfrac{-25}{17} & \sfrac{10}{27} & \sfrac{-9}{7} & \sfrac{17}{22} \\
    \sfrac{-25}{17} & -1 & \sfrac{-3}{7} & \sfrac{-7}{8} & -1 \\
    \sfrac{10}{27} & \sfrac{-3}{7} & 1 & \sfrac{-22}{7} & \sfrac{-7}{4} \\
    \sfrac{-9}{7} & \sfrac{-7}{8} & \sfrac{-22}{7} & 1 & \sfrac{8}{5} \\
    \sfrac{17}{22} & -1 & \sfrac{-7}{4} & \sfrac{8}{5} & 1
  \end{pmatrix}.
\]
This shows that the partial sign pattern $s_*(K)$, for all $|K| \le 4$,
is consistent and hence if $s_*$ is not representable, then the
inconsistency arises from the constraint $\det \Sigma > 0$.

\begin{conjecture} \label{conj:minimal-forbidden-minor}
For any $\Sigma \in \PR_5$, if $\signs_\Sigma$ agrees with $s_*$ on all
sets of sizes $1$ and $3$, then $\det \Sigma < 0$. In particular, $s_*$
is not representable.
\end{conjecture}

The number of negative eigenvalues of $\Sigma$ depends only on its sign
pattern, since it equals the number of sign changes in the principal minors
along the complete flag $\emptyset, 1, 12, 123, \dots, N$.
Thus the claim $\det \Sigma < 0$ in \Cref{conj:minimal-forbidden-minor}
is equivalent to $\Sigma$ having exactly $3$ negative eigenvalues.

The~Positivstellensatz \cite[Proposition~4.4.1]{RealAlgebra} implies that
\Cref{conj:minimal-forbidden-minor} is true if and only if there exists a
polynomial identity $-q \cdot \det \Sigma = p + h^2$ where $p,q$ are elements
of the preorder generated by the polynomials $s_*(I) \cdot \pr{I:\Sigma}$,
for $I \subseteq 12345$ with $|I| \in \{1,3\}$, in $\BB Z[\Sigma]$ and $h$
is a product of principal minors.

\subsection{Topology for reducible sign patterns}

The representability of a sign pattern $s$ gives the most basic geometric
information about $\PR_N(s)$: whether it is empty or not. At the next step,
one may be interested in the topological type of~$\PR_N(s)$. We have already
seen in \Cref{sec:ActionRep} that equivalent sign patterns under any of the
given group actions have homeomorphic representation spaces, and that the
number of connected components can be bounded in terms of the minors of~$s$.
We have one more result of this type:

\begin{definition}
A sign pattern $s$ over $N$ is \emph{reducible} if there exists a partition
$N = KL$ such that $s(I) = s|_K(I \cap K) \cdot s|_L(I \cap L)$ for all $I
\subseteq N$; otherwise $s$ is \emph{irreducible}. A \emph{decomposition}
of $s$ is a partition $N = K_1 \cdots K_m$ such that $s(I) = \prod_{i=1}^m
s|_{K_i}(I \cap K_i)$.
\end{definition}

It is easy to see that if $N = K_1 \cdots K_m = L_1 \cdots L_p$ are two
valid decompositions, then the common refinement of these partitions is
also a valid decomposition. This shows that there is a unique minimal
decomposition in which all factors $s|_{K_i}$ are necessarily irreducible.

\begin{theorem}
Let $s$ be a sign pattern over $N$ and $N = K_1 \cdots K_m$ its irreducible
decomposition. Then $s$ is representable if and only if each $s|_{K_i}$ is
representable. In this case $\dim H^0(\PR_N(s)) \ge \prod_{i=1}^m \dim
H^0(\PR_{K_i}(s|_{K_i}))$.
\end{theorem}

\begin{proof}
Clearly if $s$ is representable then so are its restrictions. Conversely,
let $\Sigma_i \in \PR_{K_i}(s|_{K_i})$. Then the block-diagonal matrix with
diagonal blocks $\Sigma_i$ is a representation for~$s$.

The same idea also gives the bound on the number of connected components.
Consider the map $\PR_N(s) \to \bigtimes_{i=1}^m \PR_{K_i}(s|_{K_i})$
which sends $\Sigma \mapsto (\Sigma_{K_i} : i = 1, \dots, m)$.
By the block-diagonal matrix construction, this is a (continuous) surjection
inducing a surjection from connected components of the domain onto those of
the image.
\end{proof}

\begin{conjecture} \label{conj:Contractible}
Every connected component of $\PR_N$ is contractible.
\end{conjecture}

The conjecture implies that the homotopy type of $\PR_N(s)$ depends only
on the number of its connected components. It is true for $n=3$ as can be
seen from \Cref{fig:Reps}. It is also true in another special case. Say
that \(s\) is \emph{completely reducible} if it decomposes into sets of
size \(1\).

\begin{theorem} \label{thm:StarShape}
  A sign pattern $s$ is completely reducible if and only if it is
  representable by a diagonal matrix. In this case $\PR_N(s)$ is a
  star-shaped domain (in particular contractible).
\end{theorem}
\begin{proof}
  If $s$ is completely reducible, then it is representable by the diagonal matrix
  $D$ with diagonal entries $\pm1$ as dictated by the signs $s(i)$, for $i \in N$.
  For the converse it suffices to observe that the sign pattern of any diagonal
  matrix is completely reducible.

  Take any \(\Sigma\) in \(\PR_N(s)\). We will show that the line segment
  \(\gamma(t) = tD + (1-t)\Sigma\) with \(t \in [0,1]\) is entirely contained
  in \(\PR_N(s)\).
  By using the multilinearity of the determinant for each row of $\gamma(t)$,
  one gets the following formula, in which it is crucial that \(D\) is a
  diagonal matrix:
  \[
    \det \gamma(t) = \sum_{I \subseteq N} (1-t)^{|I|} \pr{I:\Sigma} \cdot
      t^{|N \setminus I|} \pr{N \setminus I:D}.
  \]
  But both \(\Sigma\) and \(D\) agree with the completely reducible sign
  pattern \(s\), and so the sign of \(\pr{I:\Sigma}\cdot \pr{N \setminus I:D}\)
  is \(s(N)\) for all \(I\). Furthermore, for all \(t \in (0,1)\), the
  product $(1-t)^{|I|} \cdot t^{|N\setminus I|}$ is positive. Hence, every
  term in the above sum has constant sign~$s(N)$ which is thus the sign of
  the sum.
  The same argument works for any principal minor of $\gamma(t)$ showing
  that $\gamma(t) \in \PR_N(s)$.
\end{proof}

\begin{remark}
Due to \Cref{lemma:PositiveHypGrp} there is exactly one completely reducible
sign pattern up to the hyperoctahedral symmetry which we may take to have
all signs positive. Hence by \Cref{thm:ActionsTop} the representation spaces
of all completely reducible sign patterns are homeomorphic to the convex cone
$\PD_N$ of positive definite matrices and thus contractible. The remarkable
aspect of \Cref{thm:StarShape} is that all these spaces are even star-shaped.
\end{remark}

\section{Further remarks}
\label{sec:Remarks}

In this section we comment on the complexity of deciding representability
as well as two natural variants of our setup. The first one restricts
principal regularity and sign patterns to only \emph{leading principal minors}.
In this case, topology and representability turn out to be trivial.
Afterwards, we close with some remarks about the generalization to
vanishing principal minors.

\subsection{Complexity of representability testing}

Whether or not a given sign pattern is representable by a principally
regular matrix is a decision problem which, by definition, reduces to
the existence of a solution to a system of strict polynomial inequalities
with integer coefficients. The latter is polynomial-time equivalent to
a well-known decision problem in computational geometry called \TT{ETR}
(which stands for \emph{existential theory of the reals}). The decision
problems which many-one reduce to \TT{ETR} in polynomial time comprise
the complexity class $\exists\BB R$; cf.~\cite{SchaeferETR}.

\begin{definition}
A \emph{partial sign pattern} consists of a finite ground set $N$ together
with a list of pairs $(K_i, s_i)$ where $K_i \subseteq N$ and $s_i \in
\Set{\TT+,\TT-}$.
The representability problem \TT{RepPR} asks to decide for a given partial
sign pattern (in the obvious encoding) whether there exists $\Sigma \in
\PR_N$ such that $\sgn \pr{K_i:\Sigma} = s_i$ for all~$i$.
\end{definition}

\begin{proposition}
\TT{RepPR} reduces to \TT{ETR} in polynomial time, and hence $\TT{RepPR} \in \exists\BB R$.
\end{proposition}

\begin{proof}
We first show how to write out the inequality constraint $s_i \pr{K_i:\Sigma} > 0$
in time polynomial in~$|K_i|$. Note that merely writing out the determinant using
the Leibniz formula will not suffice since it has $|K_i|!$ terms. We use instead
the fact that $\Sigma_{K_i}$ is symmetric. By the spectral theorem, there exists
an orthonormal basis of eigenvectors, which we write into the columns of a matrix~$V_i$,
and a diagonal matrix of eigenvalues $D_i$ such that
\[
  \Sigma_{K_i} = V_i^\T \, D_i \, V_i.
\]
This matrix equation as well as the stipulation that it is the spectral
decomposition of $\Sigma_{K_i}$ can be written down using polynomial
equations which are polynomially sized in the number of elements in~$K_i$.
Then $\pr{K_i:\Sigma}$ is equal to the product of the diagonal elements
in~$D_i$.

Writing out these equations for all given pairs $(K_i, s_i)$ results in a
system of polynomial equations and strict inequalities whose solutions are
symmetric matrices representing the given partial sign pattern. Since this
space of matrices is open and $\PR_N$ is dense in $\Sym_N$, the existence
of a general symmetric matrix solution is equivalent to the existence of
a solution in~$\PR_N$.
\end{proof}

A decision problem is \emph{$\exists\BB R$-complete} if it is in $\exists\BB R$
and every $\exists\BB R$ problem reduces to it in polynomial time. Thus it
occupies the most difficult stratum of problems in $\exists\BB R$. In light
of the growing body of literature on $\exists\BB R$-complete problems (see
the collection \cite{CompendiumETR}), we ask:

\begin{question}
Is \TT{RepPR} $\exists\BB R$-complete?
\end{question}

\subsection{Space of leading principal minors}

Instead of $\PR_N$ requiring all principal minors to be non-zero, one could
also consider the space $\LPR_N$ of real symmetric $N \times N$ matrices
whose \emph{leading} principal minors are non-zero. Their corresponding sign
vectors are functions $\ell: \Set{0, \dots, n} \to \Set{\TT+, \TT-}$ given
by $\ell(k) = \sgn \pr{1\cdots k:\Sigma}$; they all satisfy $\ell(0) = \TT+$.
Given any sign pattern ${s: 2^N \to \Set{\TT+, \TT-}}$, its leading sign pattern
is obtained as $\ell(k) = s(\Set{1, \dots, k})$.
Hence,~each $\PR_N(s)$ is contained in exactly one of the spaces $\LPR_N(\ell)$
which only imposes leading principal~minor~signs.

However, $\LPR_N(\ell)$
also contains matrices with vanishing principal minors, as long as they are
not leading. Hence, the natural $\SymGrp_N$ action on symmetric matrices
does not induce a sensible action on leading principal minor sign patterns.
Moreover, we can show that all $\LPR_N(\ell)$ are non-empty and topologically
trivial.
In short, the reason is that $\LPR_N(\ell)$ is under-constrained as it is
defined by $n$ strict inequalities in $\binom{n+1}{2}$~unknowns. In
particular, by Schur complement, the statement $\sgn \pr{1\cdots k:\Sigma}
= \ell(k)$ can be written as
\begin{align*}
  \label{eq:Schurkk}
  \tag{$*$}
  \ell(k) \cdot p_k = \ell(k) \cdot p_{k-1} \cdot (\sigma_{kk} - h_{k-1}) > 0,
\end{align*}
where $p_k = \pr{1\cdots k:\Sigma}$, and $h_{k-1} = \Sigma_{k,1\cdots(k-1)}\,
\adj(\Sigma_{1\cdots(k-1)})\, \Sigma_{1\cdots(k-1),k}$, the latter of which
is a polynomial in variables $\sigma_{ij}$ with $(i,j) < (k,k)$ in the
lexicographic order from the right.

\begin{proposition} \label{prop:LPR}
Every leading principal minor sign pattern $\ell: \Set{0, \dots, n} \to
\Set{\TT+, \TT-}$ is representable by a matrix $\Sigma \in \PR_N$.
Moreover, its space of representations $\LPR_N(\ell)$ is homeomorphic
to $\PD_N$, the convex cone of positive definite matrices, and
therefore contractible.
\end{proposition}

\begin{proof}
Let $\ell: \Set{0, \dots, n} \to \Set{\TT+, \TT-}$ be arbitrary. It is easy
to explicitly construct a diagonal matrix with $\pm1$ on the diagonal which
represents the leading principal minor sign sequence~$\ell$. This matrix is
moreover a hyperoctahedral image of the identity matrix.

For the second claim, we construct a homeomorphism between any two $\LPR_N(\ell)$
and $\LPR_N(\ell')$ inductively, mapping $\Sigma$ to $\Sigma'$. In this map,
all off-diagonal entries of $\Sigma$ are preserved; only the diagonal
entries change, i.e., $\sigma'_{ij} = \sigma_{ij}$ for all $i \neq j$.
We use the notation $h_k$ from \eqref{eq:Schurkk} to refer to the polynomial
in $\Sigma$ and $h'_k$ for the respective polynomial in $\Sigma'$.
Setting
\[
  \sigma'_{kk} \defas h'_{k-1} + \ell(k) \cdot \ell(k-1) \cdot \ell'(k) \cdot
    \ell'(k-1) \cdot (\sigma_{kk} - h_{k-1})
\]
turns the assumption that $\sgn p_k = \ell(k)$ into the conclusion that
$\sgn p'_k = \ell'(k)$ via \eqref{eq:Schurkk} and using that, by induction,
$\sgn p'_{k-1} = \ell'(k-1)$. Note that $h_0 = 0$ and $\sigma'_{11} =
\pm\sigma_{11}$ and so, inductively, $h'_{k-1}$, $\sigma'_{kk}$ and hence
the entire $\Sigma'$ are continuous functions of $\Sigma$. By exchanging
the roles of $\ell$ and $\ell'$ in this construction, one constructs the
inverse of this map, proving that it is indeed a homeomorphism.
\end{proof}

\begin{remark}
Viewing the inequality~\eqref{eq:Schurkk} as constraining $\sigma_{kk}$
in terms of $p_{k-1}$ and $h_{k-1}$ which are polynomials in lexicographically
lower entries of $\Sigma$, one can also show that the $\LPR_N(\ell)$ are the
cells of a cylindrical algebraic decomposition of~$\LPR_N$ with respect to
the lexicographic ordering.
\end{remark}

\subsection{Sign patterns with zeros}

A generic symmetric matrix is principally regular in the sense that the
complement $\Sym_N \setminus \PR_N$ is an affine subvariety of $\Sym_N$
of codimension~$1$. As~demonstrated in \Cref{sec:Constructions}, the
assumption that all principal minors are invertible furnishes a well-behaved
combinatorial theory with minors and duality. In applications to algebraic
statistics it has been observed that principal regularity is responsible
for other remarkable combinatorial effects: for example, the \emph{intersection
property} of Gaussian conditional independence can be proved directly
from \Cref{eq:Mdiamond} and the assumption of principal regularity
\cite[Corollary~1]{MatusGaussian}; while for singular Gaussian measures,
the combinatorial theory of conditional independence becomes much more
involved resulting, in particular, in the failure of the intersection
property; compare \cite{LnenickaMatus} and \cite{SimecekGaussian}.

On the other hand, from a matroid theory perspective, principal minor
sign patterns of principally regular matrices correspond to orientations
of only a single Lagrangian matroid --- the uniform one; as defined
in~\cite{B.B.G+2000}. By allowing principal minors to vanish, the resulting
sign patterns with image in $\Set{\TT0, \TT+, \TT-}$ represent more
Lagrangian matroids and the corresponding representation spaces decompose
the entire affine space $\Sym_N$.
We note that in this more general setting, the upper bound on connected
components (\Cref{cor:pr-h0}) remains valid because \cite[Theorem~1]{B.P.R1996}
also accounts for zero constraints on polynomials.

The combinatorial interplay of principal minor signs is more complex
when vanishing is allowed. A coarse analogue of such sign patterns
stratified by the order of principal minors is studied in \cite{sepr}
To illustrate the effect of vanishing principal minors, we quote here
the \TT{NN} Theorem of \cite{sepr}: if $\sgn \pr{K:\Sigma} = \TT0$ for
all sets $K$ with $k \le |K| \le k+1$, then this holds for all $|K| \ge k$.
It already seems interesting to determine the representable \emph{PSD}
sign patterns, i.e.,~those $s: 2^N \to \Set{\TT0, \TT+}$ which are
representable by positive semidefinite matrices. These sign patterns
give a natural stratification of the PSD cone which is relevant in
the algebraic approach to semidefinite programming \cite{AlgebraicSDP}.

\subsection*{Acknowledgements}

Tobias Boege was partially supported by the Academy of Finland grant number
323416 and by the Wallenberg Autonomous Systems and Software Program (WASP)
funded by the Knut and Alice Wallenberg Foundation. Jesse Selover was partially
supported by NSF grant DMS-2154019.
The authors are grateful for the hospitality of the Max-Planck Institute for
Mathematics in the Sciences where this project was initiated. We~would like
to thank Bernd Sturmfels for helpful suggestions.
Part of this research was performed while some of the authors were visiting
the Institute for Mathematical and Statistical Innovation (IMSI), which is
supported by the National Science Foundation (Grant No.\ DMS-1929348).
We thank Andrey Zabolotskiy for suggesting references towards \Cref{thm:Asymptotic}
which appeared as a conjecture in an earlier version of this paper, and
Claudia Yun for helpful discussions in finding a proof.

\bibliographystyle{tboege}
\bibliography{main.bib}

\let\etalchar\undefined
\nocitesoft{Mathematica, GANAK, TodaSAT, SCIP, sagemath}
\bibliographystylesoft{tboege}
\bibliographysoft{main.bib}

\enlargethispage{3em}

\end{document}